\documentclass[12pt]{article}
\usepackage{amssymb, amscd, amsthm, amsmath,latexsym, amstext,mathrsfs,verbatim}
\usepackage[all]{xy}
\usepackage{graphicx,enumerate}

\def\np{\bigskip}

\def\ddA{{\rm A}}
\def\ddD{{\rm D}}

\def\TL{{\rm TL}}
\def\STL{{\rm STL}}
\def\RTL{{\rm DTL}}
\def\RTLM{{\rm DTLM}}
\def\TLM{{\rm TLM}}

\def\Br{{\rm Br}}

\def\BrD{{\rm{ BrD}}}
\def\BrM{{\rm BrM}}

\def\ddB{{\rm B}}
\def\ddC{{\rm C}}
\def\ddD{{\rm D}}
\def\ddE{{\rm E}}
\def\ddF{{\rm F}}
\def\ddG{{\rm G}}

\def\hE{{\hat E}}
\def\eps{{\epsilon}}
\def\alp{{\alpha}}
\def\BrMD{{\rm BrMD}}

\def\lijntje{\vrule height2.4pt depth-2pt width0.5in}

\def\vlijntje{\vrule height0.45in depth0.4pt width0.4pt}
\def\vlijn{\buildrel {\hbox to 0pt{\hss$\textstyle\circ$\hss}}\over\vlijntje}

\def\dlijntje{{\vrule height2pt depth-1.6pt
width0.5in}\llap{\vrule height4pt depth-3.6pt width0.5in}}
\def\vtriple#1\over#2\over#3{\mathrel{\mathop{\kern0pt #2}\limits_{\hbox
to 0pt{\hss$#1$\hss}}^{\hbox to 0pt{\hss$#3$\hss}}}}
\def\rvtriple#1\over#2\over#3{\mathrel{\mathop{\kern0pt #2}\limits_{\hbox
to 0pt{\hss$#3$\hss}}^{\hbox to 0pt{\hss$#1$\hss}}}}
\def\tlijntje{{\vrule height1.7pt depth-1.3pt
width0.5in}\llap{\vrule height3.0pt depth-2.6pt width0.5in}\llap{\vrule height4.3pt depth-3.9pt width0.5in}
}

\def\Dm{\vtriple{\scriptstyle n+1}\over\circ\over{}\kern-1pt\lijntje\kern-1pt
\vtriple{\scriptstyle{n}}\over\circ\over{}
\cdots\cdots\vtriple{\scriptstyle 4}\over\circ\over{}\kern-1pt\lijntje\kern-1pt
\vtriple{\scriptstyle 3}\over\circ\over{\buildrel
{\scriptstyle 2}\over\vlijn}\kern-1pt\lijntje\kern-1pt
\vtriple{1}\over\circ\over{}\kern-1pt}

\def\Dn{\vtriple{\scriptstyle n}\over\circ\over{}\kern-1pt\lijntje\kern-1pt
\vtriple{\scriptstyle{n-1}}\over\circ\over{}
\cdots\cdots\vtriple{\scriptstyle 4}\over\circ\over{}\kern-1pt\lijntje\kern-1pt
\vtriple{\scriptstyle 3}\over\circ\over{\buildrel
{\scriptstyle 2}\over\vlijn}\kern-1pt\lijntje\kern-1pt
\vtriple{1}\over\circ\over{}\kern-1pt}
\def\En{\vtriple{\scriptstyle n}\over\circ\over{}\kern-1pt\lijntje\kern-1pt
\vtriple{\scriptstyle{n-1}}\over\circ\over{}
\cdots\cdots\vtriple{\scriptstyle 5}\over\circ\over{}\kern-1pt\lijntje\kern-1pt
\vtriple{\scriptstyle 4}\over\circ\over{\buildrel
{\scriptstyle 2}\over\vlijn}\kern-1pt\lijntje\kern-1pt
\vtriple{\scriptstyle 3}\over\circ\over{}\kern-1pt\lijntje\kern-1pt
\vtriple{\scriptstyle 1}\over\circ\over{}\kern-1pt}
\def\An{\vtriple{\scriptstyle n}\over\circ\over{}\kern-1pt\lijntje\kern-1pt
\vtriple{\scriptstyle{n-1}}\over\circ\over{}\kern-1pt\lijntje\kern-1pt
\vtriple{\scriptstyle n-2}\over\circ\over{}
\cdots\cdots
\vtriple{\scriptstyle 2}\over\circ\over{}\kern-1pt\lijntje\kern-1pt
\vtriple{\scriptstyle 1}\over\circ\over{}\kern-1pt}
\def\Cn{\vtriple{\scriptstyle n-1}\over\circ\over{}
\kern-1pt\lijntje\kern-1pt\vtriple{\scriptstyle{n-2}}\over\circ\over{}
\cdots\cdots
\vtriple{\scriptstyle 2}\over\circ\over{}
\kern-1pt\lijntje\kern-1pt\vtriple{\scriptstyle 1}\over\circ\over{}
\kern-4pt{\dlijntje \kern -25pt<}\kern10pt
\vtriple{\scriptstyle 0}\over\circ\over{}\kern-1pt}
\def\Ct{\vtriple{\scriptstyle 2}\over\circ\over{}
\kern-1pt\lijntje\kern-1pt\vtriple{\scriptstyle 1}\over\circ\over{}
\kern-4pt{\dlijntje \kern -25pt<}\kern12pt
\vtriple{\scriptstyle 0}\over\circ\over{}\kern-1pt
}
\def\Bn{\vtriple{\scriptstyle n-1}\over\circ\over{}
\kern-1pt\lijntje\kern-1pt\vtriple{\scriptstyle{n-2}}\over\circ\over{}
\cdots\cdots
\vtriple{\scriptstyle 2}\over\circ\over{}
\kern-1pt\lijntje\kern-1pt\vtriple{\scriptstyle 1}\over\circ\over{}
\kern-4pt{\dlijntje \kern -25pt>}\kern10pt
\vtriple{\scriptstyle 0}\over\circ\over{}\kern-1pt}
\def\Bt{\vtriple{\scriptstyle 2}\over\circ\over{}
\kern-1pt\lijntje\kern-1pt\vtriple{\scriptstyle 1}\over\circ\over{}
\kern-4pt{\dlijntje \kern -25pt>}\kern12pt
\vtriple{\scriptstyle 0}\over\circ\over{}\kern-1pt}
\def\Es{\vtriple{\scriptstyle 6}\over\circ\over{}\kern-1pt\lijntje\kern-1pt
\vtriple{\scriptstyle 5}\over\circ\over{}\kern-1pt\lijntje\kern-1pt
\vtriple{\scriptstyle 4}\over\circ\over{\buildrel
{\scriptstyle 2}\over\vlijn}\kern-1pt\lijntje\kern-1pt
\vtriple{3}\over\circ\over{}\kern-1pt\lijntje\kern-1pt
\vtriple{\scriptstyle 1}\over\circ\over{}\kern-1pt}
\def\Ff{
\vtriple{\scriptstyle 1}\over\circ\over{}
\kern-1pt\lijntje\kern-1pt\vtriple{\scriptstyle 2}\over\circ\over{}
\kern-4pt{\dlijntje \kern -25pt<}\kern10pt
\vtriple{\scriptstyle 3}\over\circ\over{}\kern-1pt\lijntje\kern-1pt
\vtriple{\scriptstyle 4}\over\circ\over{}
\kern-1pt}
\def\Ht{
\vtriple{\scriptstyle 1}\over\circ\over{}
\kern-1pt\overset{5}{\lijntje}\kern-1pt\vtriple{\scriptstyle 2}\over\circ\over{}
\kern-1pt\lijntje\kern-1pt
\vtriple{\scriptstyle 3}\over\circ\over{}\kern-1pt}
\def\Hf{
\vtriple{\scriptstyle 1}\over\circ\over{}
\kern-1pt\overset{5}{\lijntje}\kern-1pt\vtriple{\scriptstyle 2}\over\circ\over{}
\kern-1pt\lijntje\kern-1pt
\vtriple{\scriptstyle 3}\over\circ\over{}\kern-1pt\lijntje\kern-1pt
\vtriple{\scriptstyle 4}\over\circ\over{}
\kern-1pt}
\def\In{
\vtriple{\scriptstyle 0}\over\circ\over{}
\kern-1pt\overset{n}{\lijntje}\kern-1pt\vtriple{\scriptstyle 1}\over\circ\over{}
\kern-1pt}
\def\Gt{
\vtriple{\scriptstyle 0}\over\circ\over{}
\kern-4pt{\tlijntje\kern -25pt<}\kern 10pt\vtriple{\scriptstyle 1}\over\circ\over{}
\kern-1pt}
\def\EBn{\vtriple{\scriptstyle n-1}\over\circ\over{}
\kern-1pt\lijntje\kern-1pt\vtriple{\scriptstyle{n-2}}\over\circ\over{\buildrel
{\scriptstyle -1}\over\vlijn}\cdots\cdots
\vtriple{\scriptstyle 2}\over\circ\over{}
\kern-1pt\lijntje\kern-1pt\vtriple{\scriptstyle 1}\over\circ\over{}
\kern-4pt{\dlijntje \kern -25pt<}\kern8pt
\vtriple{\scriptstyle 0}\over\circ\over{}\kern-1pt}
\def\Cn{\vtriple{\scriptstyle n-1}\over\circ\over{}
\kern-1pt\lijntje\kern-1pt\vtriple{\scriptstyle{n-2}}\over\circ\over{}
\cdots\cdots
\vtriple{\scriptstyle 2}\over\circ\over{}
\kern-1pt\lijntje\kern-1pt\vtriple{\scriptstyle 1}\over\circ\over{}
\kern-4pt{\dlijntje \kern -25pt<}\kern10pt
\vtriple{\scriptstyle 0}\over\circ\over{}\kern-1pt}
\def\ECn{\vtriple{\scriptstyle -2}\over\circ\over{}
\kern-4pt{\dlijntje \kern -25pt>}\kern8pt\vtriple{\scriptstyle n-1}\over\circ\over{}
\kern-1pt\lijntje\kern-1pt\vtriple{\scriptstyle{n-2}}\over\circ\over{}
\cdots\cdots
\vtriple{\scriptstyle 2}\over\circ\over{}
\kern-1pt\lijntje\kern-1pt\vtriple{\scriptstyle 1}\over\circ\over{}
\kern-4pt{\dlijntje \kern -25pt<}\kern12pt
\vtriple{\scriptstyle 0}\over\circ\over{}\kern-1pt}
\def\Fo{\vtriple{\scriptstyle -1}\over\circ\over{}
\kern-1pt\lijntje\kern-1pt
\vtriple{\scriptstyle 1}\over\circ\over{}
\kern-1pt\lijntje\kern-1pt\vtriple{\scriptstyle 2}\over\circ\over{}
\kern-4pt{\dlijntje \kern -25pt<}\kern8pt
\vtriple{\scriptstyle 3}\over\circ\over{}\kern-1pt\lijntje\kern-1pt
\vtriple{\scriptstyle 4}\over\circ\over{}
\kern-1pt}
\def\Ft{
\vtriple{\scriptstyle 1}\over\circ\over{}
\kern-1pt\lijntje\kern-1pt\vtriple{\scriptstyle 2}\over\circ\over{}
\kern-4pt{\dlijntje \kern -25pt<}\kern8pt
\vtriple{\scriptstyle 3}\over\circ\over{}\kern-1pt\lijntje\kern-1pt
\vtriple{\scriptstyle 4}\over\circ\over{}
\kern-1pt\lijntje\kern-1pt
\vtriple{\scriptstyle -2}\over\circ\over{}
\kern-1pt}
\def\Go{\vtriple{\scriptstyle -1}\over\circ\over{}
\kern-1pt\lijntje\kern-1pt
\vtriple{\scriptstyle 0}\over\circ\over{}
\kern-4pt{\tlijntje\kern -25pt<}\kern 12pt\vtriple{\scriptstyle 1}\over\circ\over{}
\kern-1pt}
\def\Gf{
\vtriple{\scriptstyle 0}\over\circ\over{}
\kern-4pt{\tlijntje\kern -25pt<}\kern 12pt\vtriple{\scriptstyle 1}\over\circ\over{}
\kern-1pt\lijntje\kern-1pt
\vtriple{\scriptstyle -2}\over\circ\over{}
\kern-1pt}

\newcommand{\cA}{\mathcal{A}}

\newcommand{\N}{\mathbb N}

\newcommand{\R}{\mathbb R}
\newcommand{\Z}{\mathbb Z}

\newcommand{\fB}{\mathfrak{B}}

\numberwithin{equation}{section}

\newtheorem{lemma}{Lemma}[section]
\newtheorem{cor}[lemma]{Corollary}
\newtheorem{prop}[lemma]{Proposition}
\newtheorem{thm}[lemma]{Theorem}

\theoremstyle{definition}

\newtheorem{defn}[lemma]{Definition}

\theoremstyle{remark}
\newtheorem{rem}[lemma]{Remark}

\newtheorem{example}[lemma]{Example}

\def\b{\beta}

\def\alp{\alpha}

\begin{document}
\title{The Dieck-Temperley-Lieb algebras in Brauer algebras}
\author{ Shoumin Liu\footnote{The author is funded by Scientific Research Foundation for Returned Scholars, Ministry of Education of China(2015)
 and the National Natural Science Foundation of China (Grant No. 11601275, Youth Program).}}
 \date{}
\maketitle
%\mainmatter
%\setcounter{section}{-1} \tableofcontents

\begin{abstract}
In this paper, we will study the Dieck-Temlerley-Lieb algebras of type $\ddB_n$ and $\ddC_n$. We compute their ranks and describe a
basis for them by using some  results from  corresponding Brauer algebras and Temperley-Lieb algebras.
\end{abstract}
%\tableofcontents

\section{Introduction}
The Temperley-Lieb algebras play an important role in representation theory and knot theory.
The classical $TL$ algebra first came out in \cite{TL1971}, and in \cite{F1997}, Fan extended it to other types as a quotient of
Hecke algebras and described a basis for each type. Dieck have defined an  diagrammatic $TL$ algebra of type $B$ and
compute its rank  in \cite{Dieck} and \cite{Dieck2003}. Now there are some works in  $TL$ category related to $TL$ algebra, which can be
found in \cite{AM2006} and \cite{KMY2014}.   \\
The  spherical Coxeter groups are very classical and important topics in Lie theory. The irreducible spherical Coxeter groups can be
classified as simply-laced types and non-simply laced types.  Tits has described how to obtain the non-simply laced types from the
simply laced types in \cite{T1959}( also see \cite{Car}), which has been applied in many related fields.
From studying the invariant theory for orthogonal groups,
Brauer discovered  Brauer algebras of type $\ddA$(\cite{Brauer1937});  Cohen, Frenk and Wales extended it to the
definition of simply laced type in \cite{CFW2008}, including type $\ddD_n$. In \cite{CLY2010} and \cite{CL2011}, Cohen, Liu and
Yu applied the similar method to obtain the Brauer algebras of type $\ddB_n$ and $\ddC_n$. In \cite{CLY2010} and \cite{CL2011},
the Brauer algebras of type $\ddC_n$ and $\ddB_n$ are described as the subalgebras of Brauer algebras of type $\ddA_{2n-1}$ and
$\ddD_{n+1}$ spanned by the submonoids invariant under the classical Dynkin diagram automorphisms, repectively.\\
%Carter described how to get Coxeter group of type $\ddG_2$ by twisting Coxeter group of  type $\ddD_4$ in \cite{Car}.
In this paper, we will apply  the same method on $TL$ algebras of type $\ddA_{2n-1}$ and
$\ddD_{n+1}$ to obtain the  $TL$ algebras of type $\ddC_n$ and $\ddB_n$. And our $TL$ algebras of type $\ddC_n$
coincide with the $TL$ algebras of type $\ddB_n$ given by Dieck. Therefore we call our $TL$ algebras Dieck-Temperley-Lieb algebras ($\RTL$ in short).\\
In this paper, we first recall some results about Brauer algebras of simply-laced types and obtain some results
about the height functions under the diagram automorphisms.
We apply these to the $TL$ algebras of simply-laced types which can be considered as subalgebras
of corresponding Brauer algebras. We prove the isomorphism  between $\RTL(\ddC_n)$ and a subalgebra of $\TL(\ddA_{2n-1})$.
We describe a rewriting  forms for  $\RTL(\ddB_n)$
 and compute the rank of  $\RTL(\ddB_n)$ being $C_n+C_{n+1}-1$, where $C_n$ and $C_{n+1}$ are Catalan numbers.\\

  The paper is sketched as the following. In Section \ref{sect:BrSL},
  we recall the definition of Brauer algebras of simply-laced types and the classical
  Brauer algebras. In Section \ref{sect:Height}, we prove that the classical diagram automorphisms to obtain
  non-simply laced Brauer algebras from simply-laced Brauer algebras are height-invariant automorphisms.
  In Section \ref{sect:DTL}, we present the definition of Dieck-Temperley-Lieb algebras, and show some algebra homomorphisms betwwen the
  $\RTL$ algebras and some corresponding Brauer algebras. We recall some notions and  results about Brauer algebras
  of simply-laced type, such as admissible root sets, Brauer monoid actions and rewriting forms in Section \ref{sect:concls}. In Section
  \ref{sect:IsoCn}, we use a combinatorial method to prove  the isomorphism of $\RTL(\ddC_n)$ and $\STL(\ddA_{2n-1})$.
  In Section \ref{sect:forms}, we present a rewriting forms for $\RTL(\ddB_n)$. In Section \ref{sect:rank}, we recall  the diagram representations for
  Brauer algebras of type $\ddD_n$, and show an diagram representation for $\RTL(\ddB_n)$ to prove that the rewriting forms  in
  Section \ref{sect:forms} are linearly independent and $\RTL(\ddB_n)$ has rank $C_n+C_{n+1}-1$.

\section{The Brauer algebras of simply-laced types}\label{sect:BrSL}

First we recall the definition of simply-laced Brauer algebra from \cite{CFW2008}.
In order to avoid confusion
with the  generators in Section \ref{sect:DTL}, the symbols of \cite{CFW2008} have been capitalized.

\begin{defn}\label{1.1}
Let $Q$ be a graph. The Brauer monoid $\BrM(Q)$ is the monoid
generated by the symbols $R_i$ and $E_i$, for  each node $i$ of $Q$ and $\delta$,
$\delta^{-1}$ subject to the following relation, where
$\sim$ denotes adjacency between nodes of $Q$.

\begin{equation}\delta\delta^{-1}=1     \label{1.1.1}
\end{equation}
\begin{equation}R_{i}^{2}=1          \label{1.1.2}
\end{equation}
\begin{equation}R_iE_i=E_iR_i=E_i     \label{1.1.3}
\end{equation}
\begin{equation}E_{i}^{2}=\delta E_{i}   \label{1.1.4}
\end{equation}
\begin{equation}R_iR_j=R_jR_i, \,\, \mbox{for}\, \it{i\nsim j} \label{1.1.5}
\end{equation}
\begin{equation}E_iR_j=R_jE_i,\,\, \mbox{for}\, \it{i\nsim j}  \label{1.1.6}
\end{equation}
\begin{equation}E_iE_j=E_jE_i,\,\, \mbox{for}\, \it{i\nsim j}    \label{1.1.7}
\end{equation}
\begin{equation}R_iR_jR_i=R_jR_iR_j, \,\, \mbox{for}\, \it{i\sim j}  \label{1.1.8}
\end{equation}
\begin{equation}R_jR_iE_j=E_iE_j ,\,\, \mbox{for}\, \it{i\sim j}       \label{1.1.9}
\end{equation}
\begin{equation}R_iE_jR_i=R_jE_iR_j ,\,\, \mbox{for}\, \it{i\sim j}     \label{1.1.10}
\end{equation}
The Brauer algebra $\Br(Q)$ is  the the free $\Z$-algebra for Brauer monoid  $\BrM(Q)$.
 \end{defn}
 The Brauer algebras $\Br(Q)$ has been well studied in \cite{CFW2008}, where the basis and ranks
 of finite types are given. Usually we call $R_i$s Coxeter generators, and  $E_i$s  Temperley-Lieb generators.

 \begin{rem} From \cite{CFW2008},
we know that the following relations
in $\Br(\ddA_m)$ hold for $i\sim j \sim k$.
\begin{eqnarray}
E_iR_jR_i&=&E_iE_j \label{3.1.1}
\\
R_jE_iE_j &=& R_i E_j  \label{3.1.2}
\\
E_iR_jE_i &=& E_i       \label{3.1.3}
\\
E_jE_iR_j &=& E_j R_i     \label{3.1.4}
\\
E_iE_jE_i &=& E_i      \label{3.1.5}
\\
E_jE_iR_k E_j &=& E_jR_iE_k E_j      \label{3.1.6}
\\
E_jR_iR_k E_j &=& E_jE_iE_k E_j           \label{3.1.7}
\end{eqnarray}
\end{rem}
\begin{rem}
In \cite{Brauer1937}, Brauer shows a diagram description for a basis of
Brauer monoid of type $\ddA_m$, which is just a diagram monoid with $2m+2$
dots and $m+1$ strands, and for each dot, there is a unique strand
connecting it with another dot. Here we suppose the $2m+2$ dots have coordinates
$(i, 0)$ and $(i,1)$ in $\R^{2}$ with $1\leq i\leq m+1$. The multiplication
of two diagrams is given by concatenation, where any closed loops formed are
replaced by a factor of $\delta$, and we give one example in Figure \ref{Fig:mult}.
The generators of $\BrM(\ddA_m)$ of the
form $R_i$ and $E_i$ correspond to the diagrams indicated in Figure
\ref{fig:gens_brauerA}. Each Brauer diagram can be written as a product of elements from
$\{R_i,E_i\}_{i=1}^m$. Henceforth, we identify $\BrM(\ddA_{m})$ with its diagrammatic version. It
makes clear that $\Br(\ddA_m)$ is a free algebra over $\Z[\delta^{\pm1}]$ of
rank $(m+1)!!$, the product of the first $m+1$ odd integers. The monomials
of $\BrM(\ddA_m)$ that correspond to diagrams
will be referred to as diagrams. 
\begin{figure}[!htb]
\begin{center}
\includegraphics[width=.7\textwidth,height=.3\textheight]{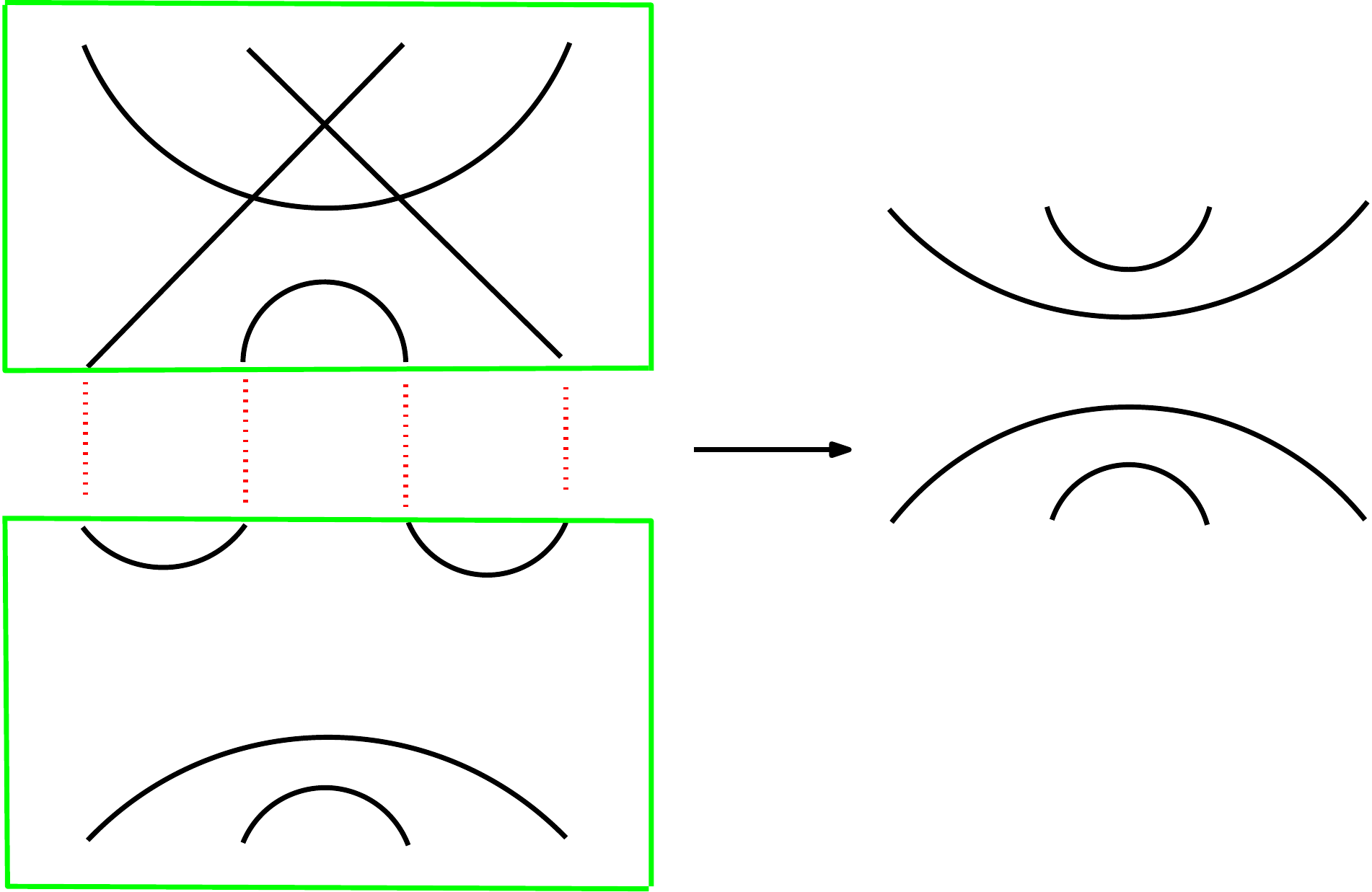}
\end{center}
\caption{One multiplication example}
\label{Fig:mult}
\end{figure}
\begin{figure}[!htb]
\begin{center}
\includegraphics[width=.7\textwidth,height=.3\textheight]{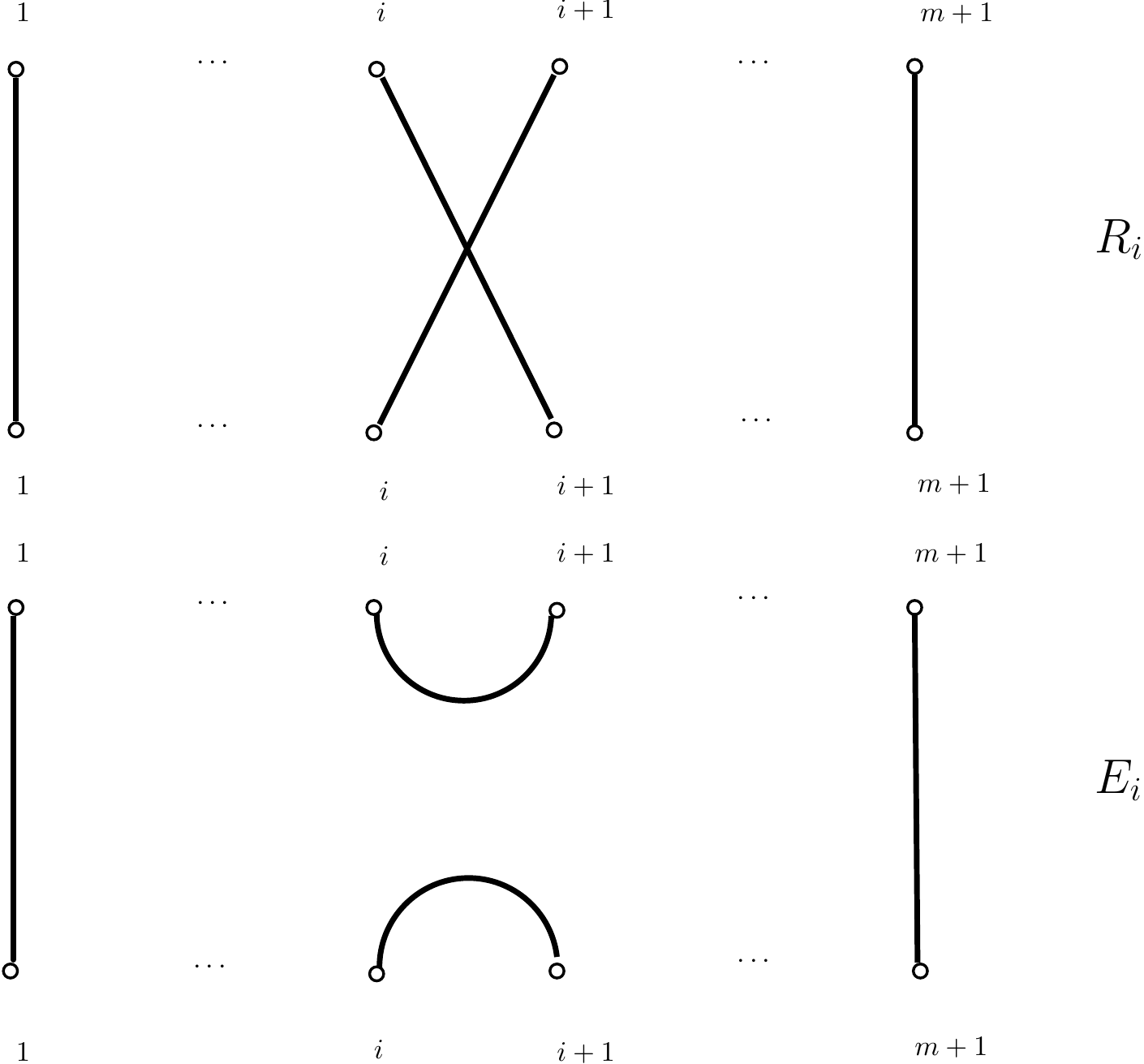}
\end{center}
\caption{Brauer diagrams corresponding to $R_i$ and $E_i$, respectively.}
\label{fig:gens_brauerA}
\end{figure}
%\begin{figure}[!htb]
%\begin{minipage}[h!]{5cm}
%\includegraphics[width=2.5\textwidth,height=2\textheight]{RbrauerA}
%\end{minipage}\phantom{aaaa}
%\begin{minipage}[h!]{5cm}
%\includegraphics[width=2.5\textwidth,height=2\textheight]{EbrauerA}
%\end{minipage}
%\vspace{-13cm}
%\caption{Brauer diagrams corresponding to $R_i$ and $E_i$, respectively.}
%\label{fig:gens_brauerA}
%\end{figure}
\end{rem}

\section{Height and automorphisms}\label{sect:Height}
\begin{rem}
We keep  notation as  in \cite[Section 2]{CW2011} and first introduce some basic conceptions.
Let $Q$ be  the  diagram  of a  connected finite simply laced Coxeter
group (type $\ddA_n$, $\ddD_n$, $\ddE_6$,  $\ddE_7$,  $\ddE_8$), and
$\BrM(Q)$ is the associated Brauer monoid as Definition \ref{1.1}.
An element  $a\in \BrM(Q)$ is said to be of \emph{height} $t$ if the minimal number of
 $R_i$  occurring in an expression of $a$ is $t$, denoted by $\rm{ht}$$(a)$.
 \end{rem}
 \begin{prop}\label{prop:heightinv}
 Let $Q$ be  the  diagram  of a  connected finite simply laced Coxeter
group  and $\BrM(Q)$ is the associated Brauer monoid.
 Let $\sigma$ be
 an automorphism on  $\BrM(Q)$, which is induced by permutation on Coxeter generators $R_i$ and
 Temperley-Lieb generators $E_i$, respectively. Then for each monomial $a\in \BrM(Q)$,
 it follws that
 $$\rm{ht(a)}=\rm{ht(\sigma(a))}.$$
\end{prop}
\begin{proof} Let $a\in \BrM(Q)$  and $t=\rm{ht}$$(a)$.
So $a$ has a reduced word which exactly has $t$  Coxeter generators $R_{i_1}$,$\dots$, $R_{i_t}$. Therefore  it follows that
$\sigma(a)$ can be reduced to a word with exactly  $t$  Coxeter generators $\sigma(R_{i_1})$,$\dots$, $\sigma(R_{i_t})$.
Therefore $\rm{ht(a)}\geq \rm{ht(\sigma(a))}$. Because  $\sigma$ is an automorphism, we can obtain that
$\rm{ht(a)}\leq \rm{ht(\sigma(a))}$. Hence $\rm{ht(a)}=\rm{ht(\sigma(a))}$.

\end{proof}
\begin{rem}\label{rem:all}
In classical finite Weyl groups, we can define  automorphisms on Dynkin diagrams of simply-laced Weyl
groups to obtain the Weyl groups of non-simply laced types listed in Figure \ref{ABCD}, and we have already applied these automorphisms on simply-laced Brauer
algebras to define and study Brauer algebras of non-simply laced types, which can be found in \cite{CLY2010} for type $\ddC_n$ from
$\ddA_{2n-1}$, \cite{CL2011}  for type $\ddB_n$ from
$\ddD_{n+1}$,
\cite{L2013} for type $\ddF_4$ from
$\ddE_6$, \cite{L2015} for type $\ddG_2$ from
$\ddD_4$. These conclusions are contained in \cite{Lthesis} for completing the project of obtaining  Brauer algebras of non-simply laced type
from simply-laced types.
\end{rem}
%Let $Q=\ddA_{2n-1}$,  or $\ddD_{n+1}$, or $\ddE_6$, or $\ddD_4$, and $M=\ddC_n$,   $\ddB_n$, $\dd$
Let $M$ and $Q$ be the non-simply laced type and simply-laced type, respectively, from the above remark, and we list them  in  the table below and the diagram automorphisms.\\
\begin{center}
\begin{tabular}{|c|c|}
  \hline
  $M$ & $Q$ \\
  \hline
  $\ddC_n$ & $\ddA_{2n-1}$ \\
  \hline
  $\ddB_n$ &$\ddD_{n+1}$ \\
  \hline
  $\ddF_4$ & $\ddE_6$ \\
  \hline
  $\ddG_2$ & $\ddD_4$ \\
  \hline
\end{tabular}
\end{center}
\begin{figure}[!htb]
\begin{center}
\includegraphics[width=.7\textwidth,height=.3\textheight]{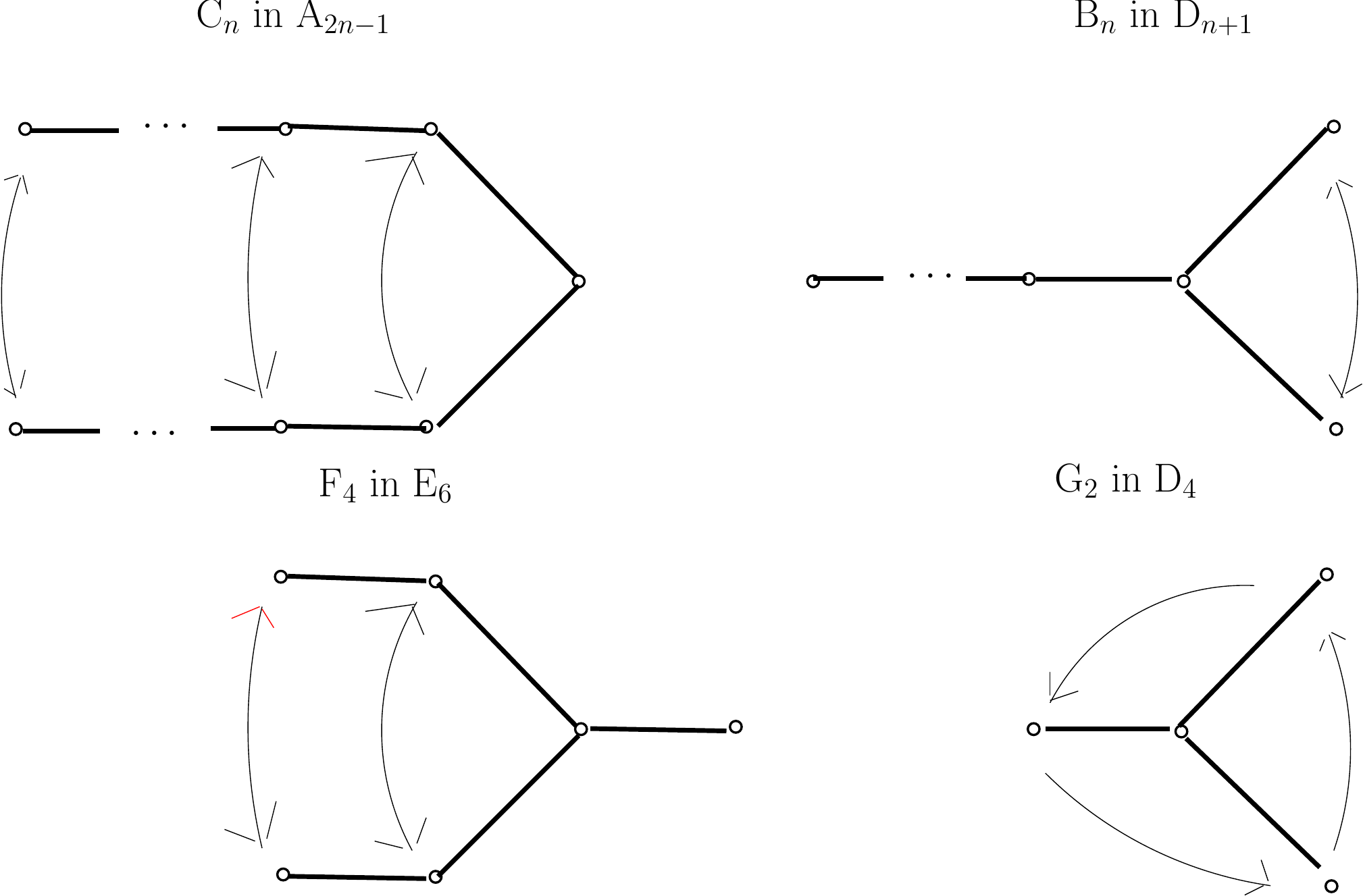}
\end{center}
\caption{The diagram automorphisms}
\label{ABCD}
\end{figure}
In the \cite{Lthesis}, to obtain the Brauer algebra of type $M$($\Br(M)$) from Brauer algebra of type $Q$
($\Br(Q)$), we always define an automorphism $\sigma$ on $\Br(Q)$ which just extends the classical automorphism on
Weyl groups on to the Temperley-Lieb generators $E_i$s, and implies that $\sigma(E_i)=E_j$ if $\sigma(R_i)=R_j$.
By Proposition \ref{prop:heightinv}, the following holds.
\begin{cor} The automorphism $\sigma$ on $\Br(Q)$ in  Remark \ref{rem:all} is a height invariant automorphism.
\end{cor}
\section{The  Dieck-Temperley-Lieb algebras}\label{sect:DTL}
Let $M$ be a connected double laced or simply laced Dynkin diagram of  finite type, namely type $\ddA_n$, $\ddB_n$, $\ddC_n$,
$\ddD_n$, $\ddE_n$($n=6$, $7$, $8$), $\ddF_4$.
We  list their Dynkin  diagrams in
 Table \ref{DKdiagram}.
 \begin{table}[!htb]
\caption{Coxeter diagrams of spherical types}\label{DKdiagram}
\begin{center}
\begin{tabular}{c|c}
type&diagram\\
\hline
$\ddA_n$&$\An$\\
$\ddD_n$&$\Dn$\\
$\ddE_n$, $6\le n\le 8$&$\En$\\
$\ddB_n$&$\Bn$\\
$\ddC_n$&$\Cn$\\
$\ddF_4$&$\Ff$
\end{tabular}
\end{center}
\end{table}
\begin{defn} \label{0.1}
Let $R$ be a commutative ring with invertible element
$\delta$.  For $n\in \N$, the \emph{reduced Temperley-Lieb algebra of type $M$} over $R$
with loop parameter $\delta$, denoted by $\RTL(M,R,\delta)$, is the
$R$-algebra generated by $\{e_i\}_{i\in M}$ subject to the following relations.
For each $i\in M$,
\begin{eqnarray}
e_{i}^{2}&=&\delta^{\kappa_i}e_{i};\label{1.1}
\end{eqnarray}
for $i$, $j\in M$  not adjacent to each other, namely $\vtriple{\scriptstyle i}\over\circ\over{}\kern-1pt\quad\kern-1pt
\vtriple{\scriptstyle{j}}\over\circ\over{}\kern-1pt$ ,
\begin{eqnarray}
e_ie_j&=&e_je_i;    \label{1.2}
\end{eqnarray}
for $i$, $j\in M$  and $\vtriple{\scriptstyle i}\over\circ\over{}\kern-1pt\lijntje\kern-1pt
\vtriple{\scriptstyle{j}}\over\circ\over{}\kern-1pt$ ,
\begin{eqnarray}
%r_ir_jr_i&=&r_jr_ir_j,  \label{0.1.10}\label{c3.0.1.10} \label{c4.0.1.10} \label{c6.0.1.9}
%\\
%r_jr_ie_j&=&e_ie_j ,             \label{0.1.13}\label{c3.0.1.13} \label{c4.0.1.11}\label{c6.0.1.10}
%\\
e_ie_je_i&=&e_i;         \label{1.3}
\end{eqnarray}
for $i$, $j\in M$ and $\vtriple{\scriptstyle i}\over\circ\over{}
\kern-4pt{\dlijntje \kern -25pt>}\kern8pt
\vtriple{\scriptstyle j}\over\circ\over{}\kern-1pt$ ,
\begin{eqnarray}
e_{j}e_i e_{j}&=&\delta e_{j}.                                                  \label{1.4}
\end{eqnarray}
%for $i$, $j\in M$ and $\vtriple{\scriptstyle i}\over\circ\over{}
%\kern-4pt{\tlijntje\kern -25pt<}\kern 12pt\vtriple{\scriptstyle j}\over\circ\over{}
%\kern-1pt$ ,
%\begin{eqnarray}
%r_ie_je_i&=&r_je_i,   \label{c7.0.1.7}
%\\
% e_ie_jr_i&=&e_ir_j,     \label{c7.0.1.8}
% \\
%e_jr_ie_jr_ie_j&=&e_j,                                                \label{c7.0.1.12}
%\\
%      e_jr_ie_jr_ir_j&=&e_jr_ir_jr_i,         \label{c7.0.1.13}
%\\
%e_ir_je_i&=&\delta^2 e_i                                \label{c7.0.1.14}
%\\
%r_jr_ie_jr_ie_j&=&r_ir_jr_ie_j. \label{c7.0.1.16}
%\\
%(r_jr_i)^6&=&1.    \label{c7.0.1.17}
%\end{eqnarray}
The parameter $\kappa_i\in \N$ is given  below, \\
for type  $\ddA_n$, $\ddD_n$, $\ddE_n$, $\kappa_i=1$ for $1\le i\le n$,\\
for type $\ddC_n$, $\kappa_0=1$, $\kappa_i=2$ for $1\le i\le n-1$;\\
for type $\ddB_n$, $\kappa_0=2$, $\kappa_i=1$ for $1\le i\le n-1$;\\
for type $\ddF_4$, $\kappa_1=\kappa_2=2$, $\kappa_3=\kappa_4=1$;\\
If $R = \Z[\delta^{\pm1}]$ we write
$\RTL(M)$ instead of $\RTL(M,R,\delta)$ and speak of  the
Dieck-Temperley-Lieb  algebra of type $M$.  The submonoid of the multiplicative
monoid of $\RTL(M)$ generated by $\delta$, $\delta^{-1}$ and $\{e_i\}_{i\in M}$
 is denoted by
$\RTLM(M)$. It is the monoid of monomials in $\RTL(M)$ and will be
called the Dieck-Temperley-Lieb monoid of type $M$.
\end{defn}
\begin{rem} When $Q$ is of simply laced type, it can be seen that $\RTL(Q)$ is the classical Temperley-Lieb algebra of type $Q$,
then we denote it by $\TL(Q)$ to replace   $\RTL(Q)$. Similarly, we replace $\RTLM(Q)$ by $\TLM(Q)$.  When $Q$ is of type $\ddA_n$, $\ddD_n$, $\ddE_6$, $\ddE_7$ and $\ddE_8$, the
algebra $\TL(Q)$ is a subalgebra of $\Br(Q)$, and the monomials of height $0$ form a basis of  $\TL(Q)$(\cite{CW2011}).
\end{rem}
Similar to the case for Brauer algebras,
there is a natural anti-involution on $\RTL(M,R,\delta)$  linearly induced by
$$x_1 x_2\ldots x_n\mapsto x_n\ldots x_2 x_1$$
with each $x_i$ being the generator of $\RTL(M,R,\delta)$.

\begin{prop} \label{prop:opp}
The identity map on
$\{\delta,  e_i \mid i\in M\}$ extends to a unique
anti-involution on the algebra $\RTL(M,R,\delta)$.
\end{prop}
Let's recall the definition of  Brauer algebras of double-laced types in the following from \cite{L2015}.
\begin{defn} \label{0.2}
Let $R$ be a commutative ring with invertible element
$\delta$ and $M$ be a Dynkin diagram of Weyl type.  For $n\in \N$, the \emph{Brauer algebra of type $M$} over $R$
with loop parameter $\delta$, denoted by $\Br(M,R,\delta)$, is the
$R$-algebra generated by $\{r_i,e_i\}_{i\in M}$ subject to the following relations.
For each $i\in M$,
\begin{eqnarray}
r_{i}^{2}&=&1,    \label{0.1.3} \label{c3.0.1.3}\label{c4.0.1.3} \label{c7.0.1.3}\label{c6.0.1.3}
\\
r_ie_i &= & e_ir_i \,=\, e_i,  \label{0.1.4} \label{c3.0.1.4} \label{c4.0.1.4}\label{c7.0.1.4}\label{c6.0.1.4}
\\
e_{i}^{2}&=&\delta^{\kappa_i}e_{i};\label{0.1.5}\label{0.1.6}\label{c6.0.1.5}\label{c3.0.1.5}\label{c3.0.1.6}\label{c4.0.1.5}\label{c4.0.1.6}\label{c7.0.1.5}\label{c7.0.1.6}
\end{eqnarray}
for $i$, $j\in M$  not adjacent to each other, namely $\vtriple{\scriptstyle i}\over\circ\over{}\kern-1pt\quad\kern-1pt
\vtriple{\scriptstyle{j}}\over\circ\over{}\kern-1pt$ ,
\begin{eqnarray}
r_ir_j&=&r_jr_i,   \label{0.1.7}\label{c3.0.1.7}\label{c4.0.1.7}\label{c6.0.1.6}
\\
e_ir_j&=&r_je_i,     \label{0.1.8}\label{c3.0.1.8}\label{c4.0.1.8}\label{c6.0.1.7}
\\
e_ie_j&=&e_je_i;    \label{0.1.9} \label{c3.0.1.9}\label{c4.0.1.9}  \label{c6.0.1.8}
\end{eqnarray}
for $i$, $j\in M$  and $\vtriple{\scriptstyle i}\over\circ\over{}\kern-1pt\lijntje\kern-1pt
\vtriple{\scriptstyle{j}}\over\circ\over{}\kern-1pt$ ,
\begin{eqnarray}
r_ir_jr_i&=&r_jr_ir_j,  \label{0.1.10}\label{c3.0.1.10} \label{c4.0.1.10} \label{c6.0.1.9}
\\
r_jr_ie_j&=&e_ie_j ,             \label{0.1.13}\label{c3.0.1.13} \label{c4.0.1.11}\label{c6.0.1.10}
\\
r_ie_jr_i&=&r_je_ir_j;         \label{0.1.15} \label{c3.0.1.15} \label{c4.0.1.12}\label{c6.0.1.11}
\end{eqnarray}
for $i$, $j\in M$ and $\vtriple{\scriptstyle i}\over\circ\over{}
\kern-4pt{\dlijntje \kern -25pt>}\kern12pt
\vtriple{\scriptstyle j}\over\circ\over{}\kern-1pt$ ,
\begin{eqnarray}
r_{j}r_ir_{j}r_{i}&=&r_ir_{j}r_ir_{j},             \label{0.1.11} \label{c3.0.1.11}\label{c4.0.1.13}
 \\
r_{j}r_ie_{j}&=&r_ie_{j},                               \label{0.1.14}\label{c3.0.1.14} \label{c4.0.1.14}
\\
  r_{j}e_ir_{j}e_i&=&e_ie_{j}e_i,                                                        \label{0.1.19} \label{c3.0.1.19}\label{c4.0.1.15}
\\
(r_{j}r_ir_{j})e_i&=&e_i(r_{j}r_ir_{j}),                                                            \label{0.1.20}\label{c3.0.1.20}\label{c4.0.1.16}
\\
e_{j}r_{i}e_{j}&=&\delta e_{j},                                                \label{0.1.12}\label{c3.0.1.12} \label{c4.0.1.17}
\\
e_{j}e_ie_{j}&=&\delta e_{j}   ,                                                 \label{0.1.16}\label{c3.0.1.16}\label{c4.0.1.18}
\\
e_{j}r_i r_{j}&=&e_{j}r_{i}     ,                                                     \label{0.1.17} \label{c3.0.1.17} \label{c4.0.1.19}
\\
 e_{j}e_ir_{j}&=&e_{j}e_i        ;                                                    \label{0.1.18} \label{c3.0.1.18}\label{c4.0.1.20}
\end{eqnarray}
%for $i$, $j\in M$ and $\vtriple{\scriptstyle i}\over\circ\over{}
%\kern-4pt{\tlijntje\kern -25pt<}\kern 12pt\vtriple{\scriptstyle j}\over\circ\over{}
%\kern-1pt$ ,
%\begin{eqnarray}
%r_ie_je_i&=&r_je_i,   \label{c7.0.1.7}
%\\
% e_ie_jr_i&=&e_ir_j,     \label{c7.0.1.8}
% \\
%e_jr_ie_jr_ie_j&=&e_j,                                                \label{c7.0.1.12}
%\\
%      e_jr_ie_jr_ir_j&=&e_jr_ir_jr_i,         \label{c7.0.1.13}
%\\
%e_ir_je_i&=&\delta^2 e_i                                \label{c7.0.1.14}
%\\
%r_jr_ie_jr_ie_j&=&r_ir_jr_ie_j. \label{c7.0.1.16}
%\\
%(r_jr_i)^6&=&1.    \label{c7.0.1.17}
%\end{eqnarray}
The parameter $\kappa_i\in \N$ is the same as Definition \ref{0.1}.  \\
%for type $\ddC_n$, $\kappa_0=1$, $\kappa_i=2$ for $1\le i\le n-1$;\\
%for type $\ddB_n$, $\kappa_0=2$, $\kappa_i=1$ for $1\le i\le n-1$;\\
%for type $\ddF_4$, $\kappa_1=\kappa_2=2$, $\kappa_3=\kappa_4=1$;\\
%for type $\ddG_2$, $\kappa_0=3$, $\kappa_1=1$.\\
If $R = \Z[\delta^{\pm1}]$ we write
$\Br(M)$ instead of $\Br(M,R,\delta)$ and speak of  the
Brauer algebra of type $M$.  The submonoid of the multiplicative
monoid of $\Br(M)$ generated by $\delta$, $\delta^{-1}$ and $\{r_i,e_i\}_{i\in M}$
 is denoted by
$\BrM(M)$. It is the monoid of monomials in $\Br(M)$ and will be
called the Brauer monoid of type $M$.
\end{defn}
By easy verification and  \cite[Table 1]{L2015}, we can obtain the conclusion below.
\begin{thm}\label{4imbeddings}
There is an algebra homomorphism $\phi: \RTL(M)\rightarrow \Br(M)$ by mapping the generators $e_i$ of
 $\RTL(M)$ to  the generators $e_i$ of
 $\Br(M)$. Furthermore we have the commutative diagram  below.
 \begin{center}
\phantom{longgggg}
\quad
\xymatrix{
 &\Br(M)\ar[rd]&\\
  \RTL(M)\ar[ru]\ar[rd] &&\Br(Q)\\
    &\TL(Q)\ar[ru]&      }
\end{center}
\end{thm}

\begin{rem}\label{Dieck}
Now we consider the case  $M=\ddC_n$, and $Q=\ddA_{2n-1}$. By \cite{Dieck}, up to some parameter,
the algebra $\RTL(\ddC_n)$ is isomorphic to the Temperley-Lieb algebras of type $\ddB_n$ defined by the author, where the rank and the diagram
representation are given. From \cite{Dieck}, the algebra $\RTL(\ddC_n)$ has rank $\binom{2n}{n}$, and the diagram generators of
$\RTL(\ddC_n)$ are given in Figure \ref{Cgenerators} with the multiplication laws as the classical Brauer algebras. Then by the diagram version of
\cite[Theorem 1.1]{CLY2010}, we find that the four morphisms in the commutative diagrams of  Theorem  \ref{4imbeddings} are injective.
\begin{figure}[!htb]
\begin{center}
\includegraphics[width=.7\textwidth,height=.3\textheight]{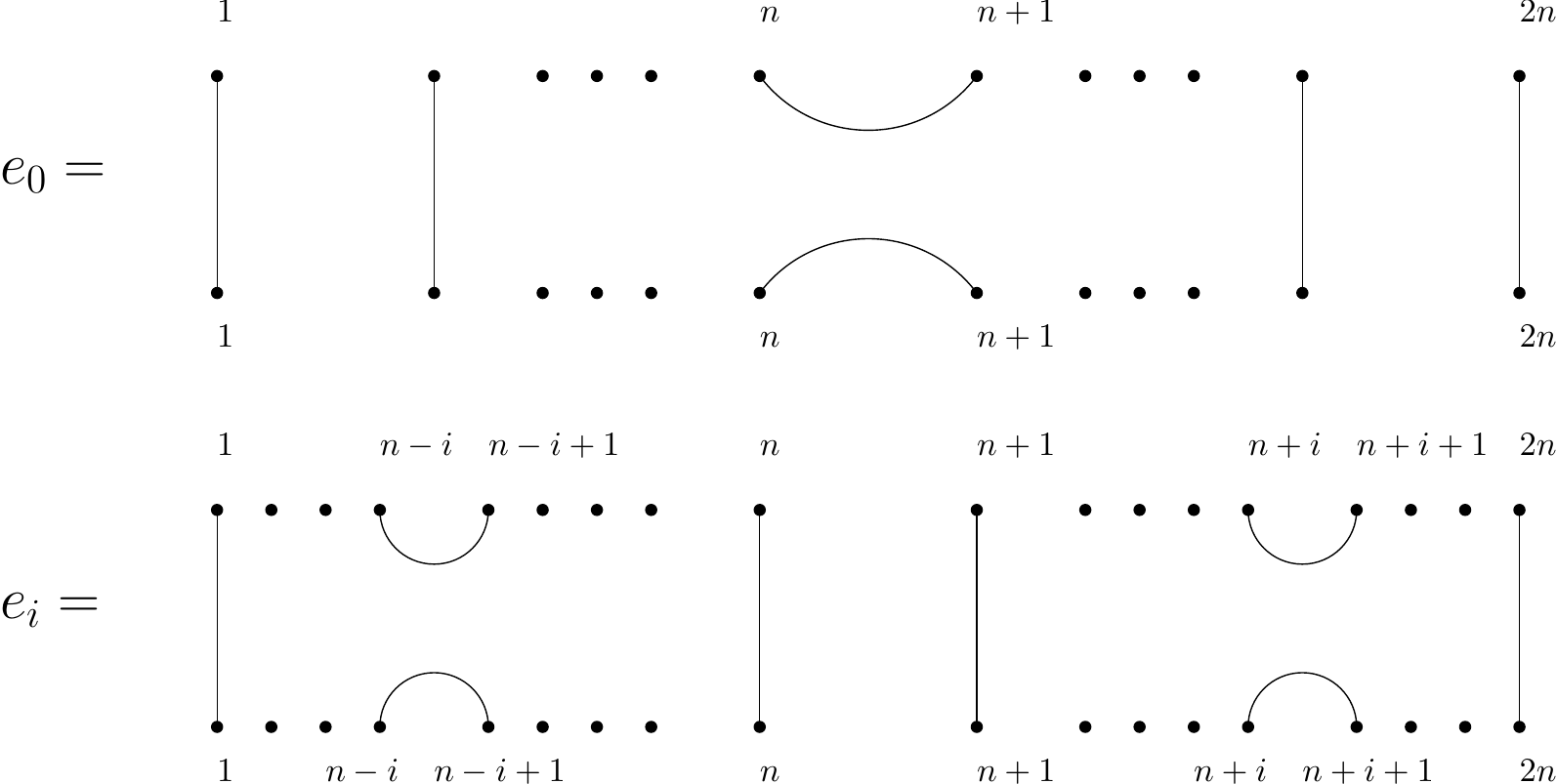}
\end{center}
\caption{The diagram generators of $\RTL(\ddC_n)$}
\label{Cgenerators}
\end{figure}
\end{rem}
As in \cite{CLY2010,CL2011,L2013,L2015}, The following can be verified.
\begin{cor} The diagram automorphisms in Figure \ref{ABCD} induce automorphisms on the corresponding simply-laced Temperley-Lieb algebras.
\end{cor}
As in \cite{CLY2010,CL2011,L2013,L2015}, we define $\STL(Q)$ is the subalgebra of $\TL(Q)$ generated by the   $\sigma$-invariant submomoid of $\TLM(Q)$, where
$Q$ can be $\ddA_{2n-1}$, $\ddD_{n+1}$ or $\ddE_6$.
\section{Some conclusions of Brauer algebras of simply-laced type}\label{sect:concls}
Let $Q$ be a spherical Coxeter diagram of simply laced type, i.e., its connected components are of type
$\ddA$, $\ddD$, $\ddE$ as listed in Table \ref{DKdiagram}. This section is to summarize some results in \cite{CGW2006}.

When $Q$ is  $\ddA_n$, $\ddD_n$, $\ddE_6$, $\ddE_7$, or $\ddE_8$, we denote it as $Q\in{\rm ADE}$.
Let $(W, T)$ be the Coxeter system of type $Q$ with $T=\{R_1,\ldots,R_n\}$ associated to the diagram of
$Q$ in Table \ref{DKdiagram}.
Let $\Phi$ be the root system of type $Q$,  let $\Phi^+$ be its positive root system, and let $\alpha_i$ be the simple root
associated to the node $i$ of $Q$. We are interested in sets $B$ of mutually commuting reflections, which has a bijective correspondence with sets of
mutually orthogonal roots of $\Phi^+$, since each reflection in $W$ is uniquely determined by a positive root and vice versa.
\begin{rem}\label{rem:positiveaction}
The action of $w\in W$ on $B$ is given by conjugation in case $B$ is described by reflections and given by
$w\{\beta_1,\ldots, \beta_p\}=\Phi^+\cap \{\pm w\beta_1,\ldots, \pm w\beta_p \}$, in case $B$ is described by positive roots.
For example, $R_4R_1R_2R_1\{\alpha_1+\alpha_2, \alpha_4\}=\{\alpha_1+\alpha_2,\alpha_4\}$, where $Q=\ddA_4$.
\end{rem}
For $\alpha$, $\beta\in \Phi$, we write $\alpha\sim\beta$ to denote $|(\alpha,\beta)|=1$. Thus, for $i$ and $j$ nodes of
$Q$, we have $\alpha_i\sim\alpha_j$ if and only if $i\sim j$.
\begin{defn}
Let $\mathfrak{B}$ be a   $W$-orbit of sets of  mutually orthogonal positive roots. We say that
$\mathfrak{B}$ is an \emph{admissible orbit} if for each $B\in \mathfrak{B}$, and $i$, $j\in Q$ with $i\not\sim j$
 and $\gamma$, $\gamma-\alpha_i+\alpha_j\in B$ we have $r_iB=r_jB$, and each element in $\mathfrak{B}$ is called an admissible root set.
\end{defn}
This is the definition from \cite{CGW2006},  and there is another equivalent definition in \cite{CFW2008}. We also state it here.
\begin{defn}
 Let $B\subset\Phi^+$ be a mutually orthogonal root set. If for all $\gamma_1$, $\gamma_2$, $\gamma_3\in B$
 and $\gamma\in \Phi^+$, with $(\gamma,\gamma_i)=1$, for
 $i=1$, $2$, $3$, we have $2\gamma+\gamma_1+\gamma_2+\gamma_3\in B$, then $B$ is called an admissible root set.
 \end{defn}
 By these two definitions, it follows that the intersection of two admissible root sets are admissible.
 It can be checked by definition that the intersection of two admissible sets are still admissible.  Hence
for a given  set $X$ of mutually orthogonal positive roots, the unique smallest admissible set containing
$X$ is called the admissible closure of $X$, and denoted as $X^{\rm cl}$ (or $\overline{X}$). Up to
the action of the corresponding Weyl groups,
all admissible root sets of type $\ddA_n$, $\ddD_n$,  $\ddE_6$, $\ddE_7$, $\ddE_8$ have appeared in  \cite{CFW2008}, \cite{CGW2009} and \cite{CW2011},
and are listed
in  Table \ref{table:admADE}. In the table, the
set
$Y(t)^*$  consists of all $\alpha^*$ for  $\alpha\in Y(t)$, where $\alpha^*$ is    the unique
positive root orthogonal to $\alpha$ and all other positive roots orthogonal
to $\alpha$ for type $\ddD_n$ with $n> 4$.  For type $\ddD_n$, if we considier the root systems are realized in $\R^{n}$,
with $\alpha_1=\epsilon_2-\epsilon_1$, $\alpha_2=\epsilon_2+\epsilon_1$, $\alpha_i=\epsilon_i-\epsilon_{i-1}$, for $3\le i\le n$,
then $\Phi^+=\{\epsilon_j\pm\epsilon_i\}_{1\le i<j\le n}$, then  $(\epsilon_j\pm\epsilon_i)^*=\epsilon_j\mp\epsilon_i$.  For $\ddD_4$,
the $t$ can be $0$, $1$, $2$, $3$, which means the number of nods in the  coclique.
When $t=2$, although in the Dynkin diagram $\{\alpha_1,\alpha_2\}$ and $\{\alpha_1,\alpha_4\}$ are symmetric,  they are  in the
 different orbits under the  Weyl group's actions.  Then the admissible root sets for $\ddD_4$ can be written as the $W(\ddD_4)$'s
 orbits of $\emptyset$, $\{\alpha_3\}$,  $\{\alpha_1,\alpha_2\}$,  $\{\alpha_1,\alpha_4\}$, and $\{\alpha_1,\alpha_2,\alpha_4,\alpha_1+\alpha_2+\alpha_4+2\alpha_3\}.$
\begin{table}
\caption{Admissible root sets of simply laced type}
\label{table:admADE}
 \begin{center}
\begin{tabular}{|c|c|}
\hline
$Q$&representatives of orbits \, under\, $W(Q)$\\
\hline
$\ddA_n$&$\{\alpha_{2i-1}\}_{i=1}^{t}$,\,$0\le t\le \left\lfloor{(n+1)/2}\right\rfloor. $\\
\hline
$\ddD_n$&$Y(t)=\{\alpha_{n+2-2i}, \alpha_{n-2},\ldots, \alpha_{n+2-2t}\}$ \, $0\le t\le \left\lfloor{n/2}\right\rfloor. $\\
        & $\{\alpha_{n+2-2i}, \alpha_{n-2},\ldots, \alpha_{4}, \alpha_1\}$  \text{if}    $2|n$\\
        & $Y(t)\cup Y(t)^*$ \, $0\le t\le \left\lfloor{n/2}\right\rfloor$\\
\hline
$\ddE_6$ & $\emptyset$, $\{\alpha_6\}$, $\{\alpha_6, \alpha_4\}$,  $\{\alpha_6, \alpha_2, \alpha_3\}^{\rm cl}$ \\
\hline
$\ddE_7$ & $\emptyset$, $\{\alpha_7\}$, $\{\alpha_7, \alpha_5\}$,  $\{\alpha_5, \alpha_5, \alpha_2\}$, $\{\alpha_7, \alpha_2, \alpha_3\}^{\rm cl}$,
                               $\{\alpha_7, \alpha_5, \alpha_2, \alpha_3\}^{\rm cl}$   \\
\hline
$\ddE_8$ & $\emptyset$, $\{\alpha_8\}$, $\{\alpha_8, \alpha_6\}$,   $\{\alpha_8, \alpha_2, \alpha_3\}^{\rm cl}$,
                               $\{\alpha_8, \alpha_5, \alpha_2, \alpha_3\}^{\rm cl}$   \\
\hline
\end{tabular}
\end{center}
\end{table}

\begin{example} If $Q=\ddD_4$, the root set $\{\alpha_1, \alpha_2, \alpha_4\}$ is mutually orthogonal but not admissible,
and its admissible closure is $\{\alpha_1, \alpha_2, \alpha_4, \alpha_1+\alpha_2+2\alpha_3+\alpha_4\}$.
\end{example}
\begin{defn}
\label{df:cA}
Let $\cA$ denote the collection of all admissible subsets of $\Phi$ consisting of
mutually orthogonal  positive roots.
Members of $\cA$ are called admissible sets.
%Suppose that if $X\subset \Phi$ consists of mutually orthogonal roots, then the minimal
%admissible containing $X$ as a subset is called the \emph{admissible closure} of $X$, denoted as $X^{\rm cl}$.
\end{defn}
Now we consider the actions of  $R_i$ on an admissible $W$-orbit $\fB$. When $R_iB\neq B$, We say that
$R_i$ lowers $B$ if there is a root $\beta\in B$ of minimal height  among those moved by $R_i$ that
satisfies $\beta-\alpha_i\in \Phi^+$ or $R_iB<B$.
We say that $R_i$ raises $B$ if there is a root $\beta\in B$ of minimal height among
those moved by $R_i$ that satisfies $\beta+\alpha_i\in \Phi^+$ or $R_iB>B$. By this
we can set an  partial order on $\fB=WB$. The poset $(\fB, <)$ with this minimal ordering is called the monoidal poset (with respect to $W$) on $\fB$ (so $\fB$ should be admissible for the poset to be monoidal). If $\fB$ just consists of sets
of a single root, the order is determined by the canonical height function on roots. There is an important conclusion in \cite{CGW2006},  stated  below.
This theorem plays a crucial role in obtaining a basis for Brauer algebra of simply laced type in \cite{CFW2008}.
\begin{thm}\label{thm:maximal}
 There is a unique maximal element in $\fB$.
 \end{thm}
 For any $\beta\in\Phi^+$ and $i\in\{1,\ldots,n\}$, there exists a $w\in
W$ such that $\beta = w\alpha_i$. Then $R_\beta := wR_iw^{-1}$ and
$E_\beta := wE_iw^{-1}$ are well defined (this is well known from Coxeter
group theory for $R_\beta$;
see \cite[Lemma 4.2]{CFW2008} for $E_\beta$).  If
$\beta,\gamma\in\Phi^+$ are mutually orthogonal, then $E_\beta$ and
$E_\gamma$ commute (see \cite[Lemma 4.3]{CFW2008}). Hence, for $B\in\cA$, we
 define the product
\begin{eqnarray}
\label{eqn:EprodB}
E_B &=& \prod_{\beta\in B} E_\beta,
\end{eqnarray}
which is a quasi-idempotent, and the normalized version
\begin{eqnarray}
\label{eqn:EhatB}
\hE_B &=& \delta^{-|B|} E_B,
\end{eqnarray}
which is an idempotent element of the Brauer monoid.
For a mutually orthogonal root subset $X\subset \Phi^+$, we have
\begin{eqnarray}
\label{eqn:Ecloure}
E_{X^{\rm cl}}=\delta^{|X^{\rm cl}\setminus X|}E_{X}.
\end{eqnarray}
% There exists $w\in W $ and a simple root
%$\alpha_{i}$ such that $\beta=w\alpha_{i}$, for any $\beta\in \Psi^+$. We define the element
%$E_{\beta}$ of $\Br(\ddA_{2n-1})$ by
%$$E_{\beta}=wE_{i}w^{-1}.$$
%By \cite{CFW2008}, we know that $E_{\beta}$ is well defined.
%
%For an arbitrary admissible mutually orthogonal positive  root set $X\subset\Phi^{+}$, we define the element $E_{X}$ by
%$$E_{X}=\prod_{\beta\in X}E_{\beta}.$$ Note that this is well defined as the
%factors in this product commute (by \cite{CFW2008}).
Let $C_X=\{i\in Q\mid \alpha_i \perp X\}$ and
let $W(C_{X})$ be the subgroup generated by the
generators of nodes in $C_X$.
The subgroup $W(C_{X})$ is called the \emph{centralizer} of $X$.
The normalizer of $X$, denoted by $N_{X}$ can be defined as
$$N_{X} =\{w\in W\mid E_X w=w E_X\}.$$
We let $D_X$ denote a set of right coset representatives for $N_X$ in $W$.\\
In \cite[Definition 3.2]{CFW2008}, an action of the Brauer monoid
$\BrM(Q)$ on the collection $\cA$ of admissible root sets
in $\Phi^+$ was indicated below, where  $Q\in {\rm ADE}$.
\begin{defn}\label{eq:aboveaction}
There is an action of the Brauer monoid $\BrM(Q)$ on the collection
$\cA$.  The generators $R_i$ $(i=1,\ldots,n)$ act by the natural action of
Coxeter group elements on its  positive root sets as in Remark \ref{rem:positiveaction},
% where negative roots are negated so
%as to obtain positive roots,
and the element $\delta$ acts as the identity,
and the action of $E_i$ $(i=1,\ldots,n)$ is defined by
\begin{equation}
E_i B :=\begin{cases}
B & \text{if}\ \alpha_i\in B, \\
(B\cup \{\alpha_{i}\})^{\rm cl} & \text{if}\ \alpha_i\perp B,\\
R_\beta R_i B & \text{if}\ \beta\in B\setminus \alpha_{i}^{\perp}.
\end{cases}
\end{equation}
\end{defn}

We will refer to this action as the admissible set action.
This monoid action plays an important role in getting a basis of $\BrM(Q)$ in \cite{CFW2008}. For the basis, we state one  conclusion from \cite[Proposition 4.9]{CFW2008} below.
 \begin{prop}\label{ADErewform}
  Each element of the Brauer monoid $\BrM(Q)$ can be written in the form $$\delta^k uE_{X}zv,$$ where
 $X$ is the highest element from one $W$-orbit in $\cA$, $u$, $v^{-1}\in D_X$, $z\in W(C_X)$, and $k\in \Z$.
 \end{prop}
 \begin{rem}\label{height}
 There is a  more general version for simply laced types in \cite{CW2011}.
We keep  notations as  in \cite[Section 2]{CW2011} and first introduce some basic conceptions.
Let $Q$ be  the  diagram  of a  connected finite simply laced Coxeter
group (type $\ddA_n$, $\ddD_n$, $\ddE_6$,  $\ddE_7$,  $\ddE_8$).
Then $\BrM(Q)$ is the associated Brauer monoid as  in Definition \ref{1.1}. Recall
an element  $a\in \BrM(Q)$ is said to be of \emph{height} $t$ if the minimal number of
$R_i$  occurring in an expression of $a$ is $t$, denoted by $\rm{ht}$$(a)$.
  By $B_Y$ we denote
the admissible closure  of $\{\alpha_i|i\in Y\}$, where $Y$ is a coclique
of $Q$. The set $B_Y$ is a minimal element in the $W(Q)$-orbit of $B_Y$ which is endowed with a  poset  structure
induced by the partial ordering $<$
defined on $W(Q)$-orbits  in $\cA$. If $d$ is the Hasse diagram distance for $W(Q)B_Y$
from $B_Y$ to the unique maximal element, then for $B\in W(Q)B_Y$ the height of $B$, already used in Definition
notation $\rm{ht}$$(B)$, is $d-l$, where $l$ is the distance in the
Hasse diagram from $B$ to the maximal element. The Figure \ref{fig:Hassediagram} is a Hasse diagram of admissible sets of type $\ddA_4$ with $2$  mutually orthogonal positive roots.
 As indicated in Theorem \ref{thm:maximal}, the set $\{\alpha_1+\alpha_2+\alpha_3, \alpha_2+\alpha_3+\alpha_4\}$ is the maximal root set in its $W(\ddA_4)$-orbit.
\begin{figure}[!htb]
\begin{center}
\includegraphics[width=.9\textwidth,height=.5\textheight]{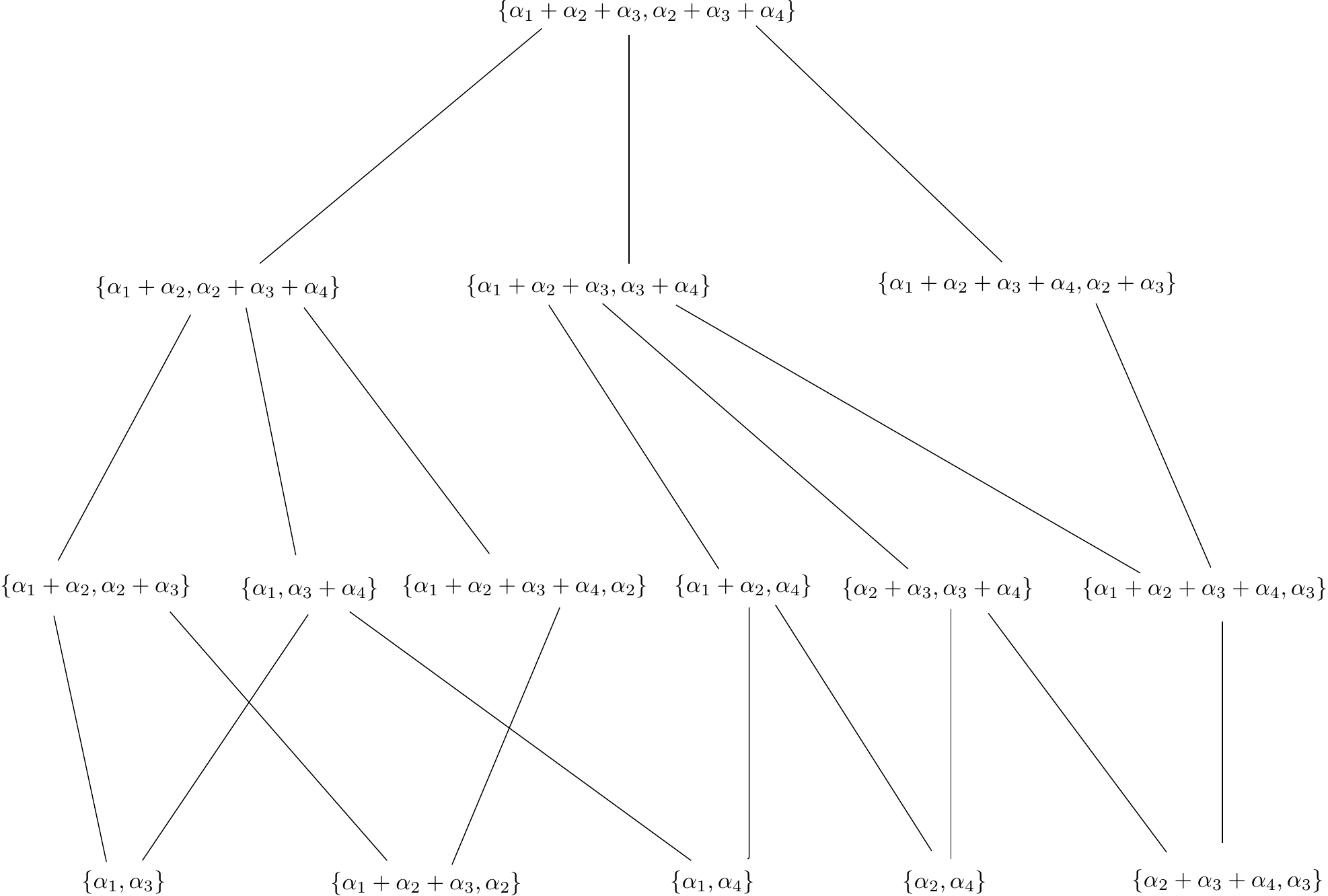}
\end{center}
\caption{A Hasse diagram of type $\ddA_4$.}
\label{fig:Hassediagram}
\end{figure}
\end{rem}
\begin{thm} \label{thm:genralwriting}
(\cite[Theorem 2.7]{CW2011}) Each monomial  $a$ in $\BrM(Q)$ can be
uniquely written as $\delta^{i} a_{B} \hat{E}_Y h a_{B'}^{\rm op}$ for some $i\in \Z$ and $h\in W(Q_{Y})$,
where $W(Q_{Y})$ is the group of invertible elements in $\hat{E_Y}W(Q)\hat{E_Y}$,
$B=a\emptyset$, $B^{'}=\emptyset a $, $a_{B}\in \BrM(Q)$, $a_{B'}^{\rm op}\in \BrM(Q)$ and
\\ (i) $a\emptyset=a_{B}\emptyset=a_{B}B_Y$,  $\emptyset a= \emptyset a_{B'}^{\rm op}= B_Y a_{B'}^{\rm op}$,
\\(ii) $\rm{ht}$$(B)=$\rm{ht}$(a_{B})$, $\rm{ht}$$(B')=$\rm{ht}$(a_{B'}^{\rm op})$.
\end{thm}

\section{The isomorphism of $\RTL(\ddC_n)$ and $\STL(\ddA_{2n-1})$}\label{sect:IsoCn}
In this section, we focus on type $\ddC_n$.
First recall the Dynkin diagram of type $\ddC_n$.
$$ \ddC_n\quad = \quad \Cn .$$
From \cite{CLY2010}, the automorphism $\sigma$ on $\BrM(\ddA_{2n-1})$ has a diagram explanation, which  means the symmetry to the middle axis.
Therefore, a $\sigma$-invariant monomial is a diagram which is symmetric to the middle axis. The same explanation also can be applied to  the
$\TLM(\ddA_{2n-1})$ , a submonoid  of  $\BrM(\ddA_{2n-1})$. Combining the diagram representation of $\TL(\ddA_{2n-1})$ and the symmetry of $\sigma$,
therefore the algebra $\STL(\ddA_{2n-1})$ has a basis consisting of the diagram monomials of $\BrM(\ddA_{2n-1})$ which have no intersections and
are symmetric to the middle axis. We give one example in the Figure \ref{fig:e1e0}.
\begin{figure}[!htb]
\begin{center}
\includegraphics[width=.5\textwidth,height=.3\textheight]{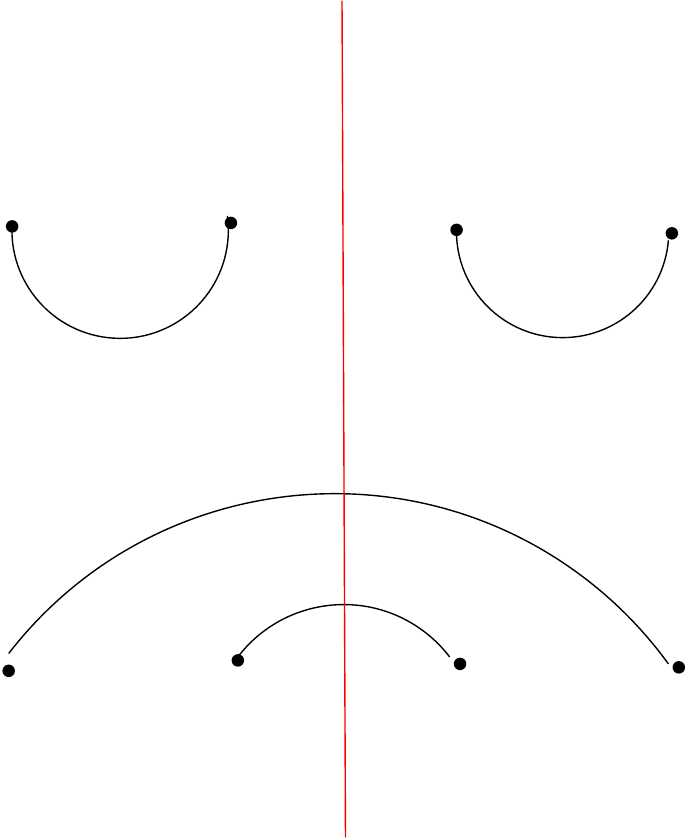}
\caption{$\phi(e_1e_0)$ in $\STL(\ddA_3)$}
\label{fig:e1e0}
\end{center}
\end{figure}
By the diagram images of generators of $\TL(\ddC_n)$,   we can obtain the following lemma.
\begin{lemma}\label{Cninto}
The algebra $\phi(\TL(\ddC_n))$  is a subalgebra of  $\STL(\ddA_{2n-1})$.
\end{lemma}
\begin{rem}
Let $m\ge1$.  The root system of the Coxeter group $W(\ddA_{m})$ of type
$\ddA_{m}$ is denoted by $\Phi$.  It is realized as $\Phi :=
\{\eps_i-\eps_j\mid 1\le i,j\le m+1,\ i\ne j\}$ in the Euclidean space
$\R^{m+1}$, where $\eps_i$ is the $i^\mathrm{th}$ standard basis vector.  Put $\alp_i
:= \eps_i-\eps_{i+1}$.  Then $\{\alpha_{i}\}_{i=1}^{m}$ is called the set of
simple roots of $\Phi$.  Denote by $\Phi^+$ the set of positive roots in
$\Phi$ with respect to these simple roots; that is, $\Phi^+ := \{\eps_i-\eps_j\mid 1\le i<j\le m+1\}$.\\
An admissible set $B$ corresponds to a Brauer diagram top in the following
way: for each $\b\in B$, where $\b = \eps_i-\eps_j$ for some
$i,j\in\{1,\ldots,m+1\}$ with $i<j$, draw a horizontal strand in the
corresponding Brauer diagram top from the dot $i$ to the dot
$j$ in the top. All horizontal strands on the top are obtained this way, so there
are precisely $|B|$ horizontal strands.\\
We will refer to this action as the {\em admissible set action}.
Alternatively, this action can be described as follows for a monomial $a$:
complete the top corresponding to $B$ into a Brauer diagram $b$, without
increasing the number of horizontal strands in the top. Now $aB$ is the top
of the Brauer diagram $ab$.\\
\end{rem}
Now we will prove the following theorem.
\begin{thm}
The algebra $\RTL(\ddC_n)$ is isomorphic to $\STL(\ddA_{2n-1})$ under $\phi$.
\end{thm}
\begin{proof}
By Remark \ref{Dieck} and Lemma \ref{Cninto}, it remains to prove that
the algebra morphism $\phi: \,\RTL(\ddC_n)\rightarrow \STL(\ddA_{2n-1})$ is surjective to accomplish  the proof.
Considering the symmetry of $\sigma$, we modify those $Y$ in the Table \ref{table:admADE} as the following.
For $i\in\{0,\ldots,n\}$ we let $Y_i\in\cA$ be the following set of nodes of
size
$i$:
$$Y_i=\begin{cases}
\{n,{n\pm2},{n\pm4},\ldots,n\pm(i-1)/2\}& \mbox{ if } i\equiv 1\pmod2\\
\{{n\pm1},{n\pm3},\ldots,n\pm i/2\}& \mbox{ if } i\equiv 0\pmod2\\
      \end{cases}
$$
The corresponding set of positive roots $\{\alpha_y\mid y\in Y\}$
is denoted by $B_i$.\\
Let $B_{Y_i}$ be the root sets corresponding to $Y_i$.
It can be seen that each $Y_i$ is of height $0$ in the sense of  Remark \ref{height}.
If we replace those $Y$s in Table \ref{table:admADE} by those $Y_i$s, which are $\sigma$-invariant and of height $0$, by the Theorem \ref{thm:genralwriting}, the proof is reduced to the
the following problem.\\
If we have an admissible set $B$ which are $\sigma$-invariant, of size $i$ and of height $0$, there exists an  element $a\in \phi(\TLM(\ddC_n))$,
such that $aY_i=B$.\\
We prove this fact  by induction. First, it can be easily verified when $n=1$, or $2$.\\
Considering the diagram version for $B$, if there is no horizontal strands having ends $i$ and $2n$(since $B$ is symmetric to the middle axis), then it is totally reduced to the case
$n-1$.
Since $B$ has height $0$, namely the diagram version of $B$ has no intersection, so there exist two possible cases for $B$ displayed in the Figure \ref{2cases}.
\begin{figure}[!htb]
\begin{center}
\includegraphics[width=.5\textwidth,height=.2\textheight]{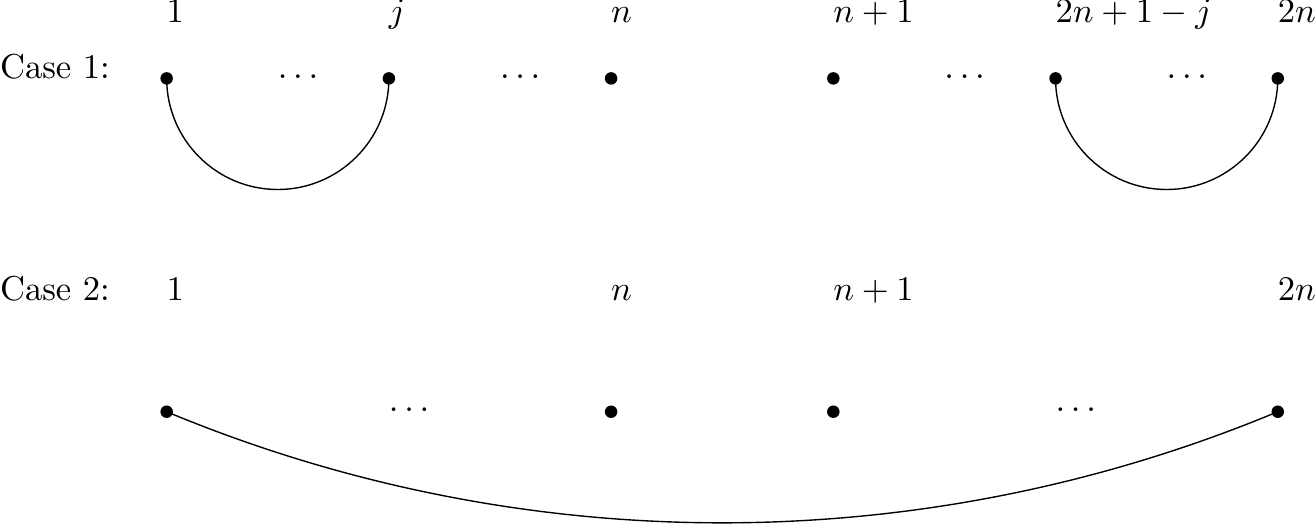}
\caption{ two possible cases}
\label{2cases}
\end{center}
\end{figure}
\\Now we can consider case $1$ in the Figure \ref{2cases}. Since $B$ has height zero,  all the strands having ends in $\{1,2,\dots, j\}$
must be horizontal, and the other ends must be also in $\{1,2,\dots, j\}$.  Then we can see that $j$ must be even. We divide the set $B$ in to two disjoint subsets $B_1$ and $B_2$,
where $B_1$ are those roots having ends in $\{1,2,\dots, j\}$ or $\{2n+1-j,2n+2-j,\dots, 2n\}$, and $B=B_1\coprod B_2$.
 It can be seen that the subalgebra $A$ of $\RTL(\ddC_n)$ generated by $\{e_i\}_{i=1}^{n-1}$ is isomorphic to $\TL(A_{n-1})$. Similarly we do the same division on
 $Y_i$, such that $Y_i=Y_{i,1}\coprod Y_{i,2}$, where $Y_{i,1}$ is the most left $\frac{j}{2}$ dots and the most right $\frac{j}{2}$ dots.
 Similarly,  $B_{Y_i}=B_{Y_{i,1}}\coprod B_{Y_{i,2}}$. By the property of $A$($\TL(\ddA_{n-1})$)(\cite[Proposition4]{CW2011}) and the symmetry to the middle axis, we can find some Temperley-Lieb monomial $a_1$ in $A$ such that
 $\phi(a_1)B_{Y_{i,1}}=B_1$, $\phi(a_1) B_{Y_{i,2}}$ and $B_2$ are lowered to the case of $\RTL(\ddC_{n-j+1})$
 which is generated by $\{e_i\}_{i=0}^{n-j}$. By indution, then we can find one monomial $a_2$ in this $\RTL(\ddC_{n-j+1})$, such that  $\phi(a_2a_1) B_{Y_{i,2}}=B_2$,
 $\phi(a_2) B_1=B_1$. Therefore $\phi(a_2 a_1)B_{Y_i}=B$.\\
 Now we consider the case $2$ in the Figure \ref{2cases}. The set $B$ has no intersection, then $|B|=n$. Let $\alpha=\epsilon_{2n}-\epsilon_1$ be
 the root represented by the
 strands from $1$ to $2n$ and  $\alpha'=\epsilon_{2n-2}-\epsilon_3$ be
 the root represented by the
 strands from $3$ to $2n-2$(Figure \ref{Cn-2}). Let $B'$ be a height $0$ and symmetric to middle axis as displayed in the top of Figure \ref{Cn-2}. By induction on
 the $\RTL(\ddC_{n-2})$ generated by $\{e_i\}_{i=0}^{n-3}$, we can see that there exist a monomial $b_1$ in this $\RTL(\ddC_{n-2})$ such that
 $\phi(b_1)B_{Y_{n-2}}=B'\setminus \{\alpha_1,\alpha_{2n-1}\}$. Therefore $\phi(b_1)B_{Y_n}=B'$. When the element $\phi(e_{n-2})$ acts on $B'$,
 we have $$\phi(e_{n-2})B'=(B'\setminus \{\alpha_1,\alpha_{2n-1},\alpha'\})\cap \{\alpha_2,\alpha_{2n-2},\alpha\}.$$
 Now we can reduce $(\phi(e_{n-2})B')\setminus\{\alpha\}$ and $B\setminus \{\alpha\} $ to  $\RTL(\ddC_{n-1})$ generated by $\{e_i\}_{i=0}^{n-2}$,
 Then we can find a monomial  $b_2$ in this $\RTL(\ddC_{n-1})$ such that $$\phi(b_2)((\phi(e_{n-2})B')\setminus\{\alpha\})=B\setminus \{\alpha\}$$and
 $$\phi(b_2e_{n-2}b_1)B_{Y_n}=B.$$
 \begin{figure}[!htb]
\begin{center}
\includegraphics[width=.5\textwidth,height=.2\textheight]{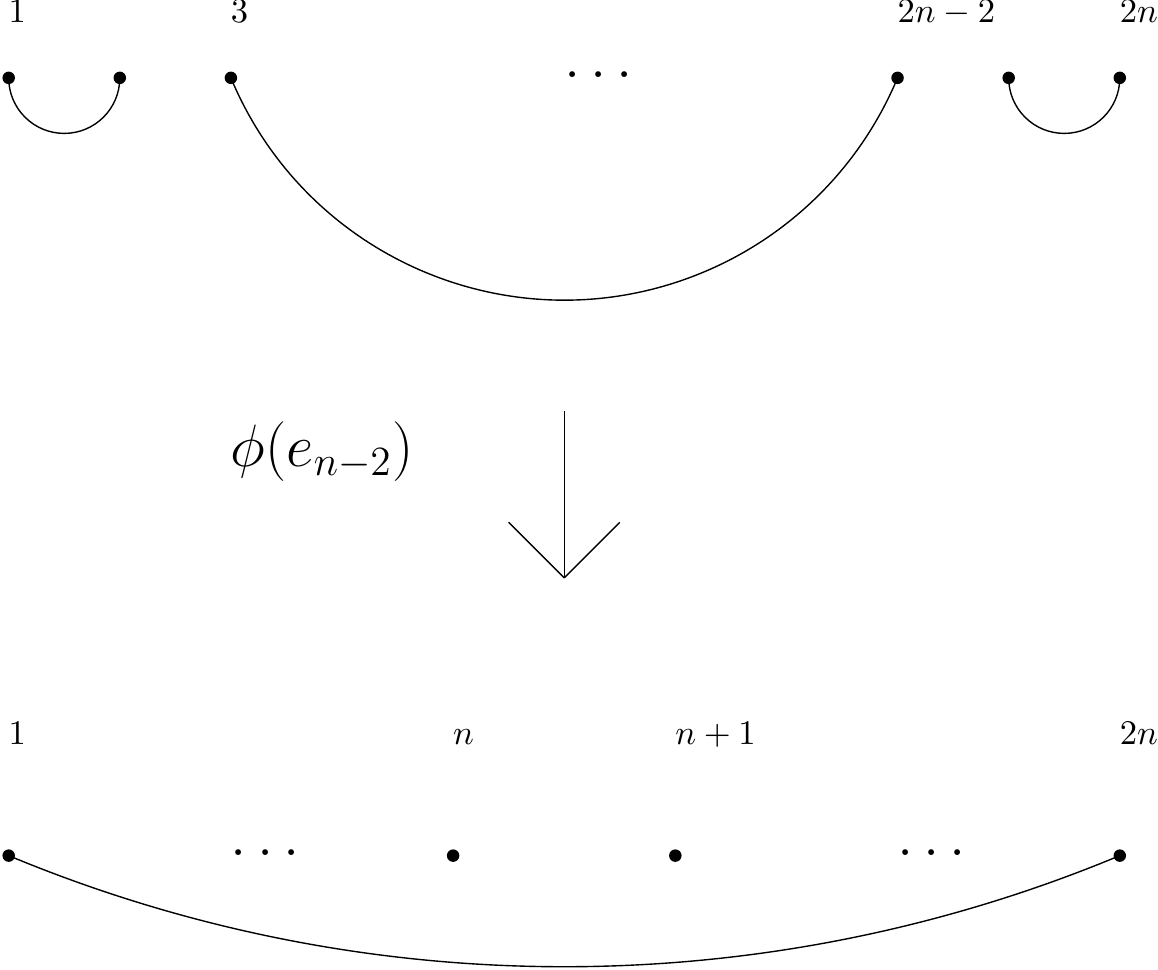}
\caption{From $\ddC_{n-2}$ to $\ddC_{n}$}
\label{Cn-2}
\end{center}
\end{figure}
\end{proof}
\section{The rewriting forms for $\RTL(\ddB_n)$}\label{sect:forms}
First recall the Dynkin diagram of type $\ddB_n$.
$$ \ddB_n\quad = \quad \Bn .$$
In \cite[Remark 6.4]{CL2011}, we have forecasted that the algebra $\RTL(\ddB_n)$ has rank $C_n+C_{n+1}-1$, where $C_n=\frac{1}{n+1}\binom{2n}{n}$, the Catalan number.
Here we will give a precise proof for this claim.
To prove this, we define
\begin{eqnarray*}
\hat{e}_0&=&e_0\\
\hat{e}_1&=&e_1e_0e_1\\
\hat{e}_{i+1}&=&e_{i+1}\hat{e}_ie_{i+1},\quad 1\leq i\leq n-2.
\end{eqnarray*}
\begin{lemma}\label{newrel}
In the algebra $\RTL(\ddB_n)$, we have the following equalities.
\begin{eqnarray}
\hat{e}_i\hat{e}_{i+1}&=&\delta \hat{e}_ie_{i+1}\quad 0\leq i\leq n-2\label{7.1.0}\\
\hat{e}_i\hat{e}_{i+1}&=&\delta e_i\hat{e}_{i+1}\quad 1\leq i\leq n-2\label{7.1.0b}\\
\hat{e}_ie_{i+1}\hat{e}_i&=&\delta\hat{e}_i,\quad 0\leq i\leq n-1\label{7.1.1}\\
\hat{e}_{i+1}e_i\hat{e}_{i+1}&=&\delta \hat{e}_{i+1}\quad 1\leq i\leq n-2\label{7.1.1b}\\
\hat{e}_i^2&=&\delta^2 \hat{e}_i,\quad 0\leq i\leq n-1\label{7.1.2}\\
\hat{e}_i\hat{e}_j&=&\hat{e}_j\hat{e}_i, \quad |i-j|>1\label{7.1.3}\\
\hat{e}_ie_j&=&e_j\hat{e}_i, \quad |i-j|>1\label{7.1.4}\\
\hat{e}_i\hat{e}_j\hat{e}_i&=&\delta^2\hat{e}_i\quad |i-j|=1\label{7.1.5}\\
\hat{e}_i\hat{e}_j&=&\delta e_i\hat{e}_j, \quad |i-j|>1\label{7.1.6}
\end{eqnarray}
\end{lemma}
\begin{proof}
For (\ref{7.1.0}), it can be verified easily when $i=0$,  and we have
\begin{eqnarray*}
\hat{e}_i\hat{e}_{i+1}
&=&e_i\hat{e}_{i-1}(e_ie_{i+1}e_i)\hat{e}_{i-1}e_ie_{i+1}\\
&\overset{(\ref{1.3})}{=}&e_i\hat{e}_{i-1}(e_i\hat{e}_{i-1}e_i)e_{i+1}=e_i\hat{e}_{i-1}\hat{e}_ie_{i+1}\\
&\overset{induction}{=}&\delta e_i\hat{e}_{i-1}e_ie_{i+1}=\delta \hat{e}_ie_{i+1}.
\end{eqnarray*}
For (\ref{7.1.0b}), we have
\begin{eqnarray*}
e_{i}\hat{e}_{i+1}&=&(e_{i}e_{i+1}e_i)\hat{e}_{i-1}e_ie_{i+1}\\
&\overset{(\ref{1.3})}{=}&e_i\hat{e}_{i-1}e_ie_{i+1}=\hat{e}_ie_{i+1}\overset{(\ref{7.1.0})}{=}\delta^{-1}\hat{e}_i\hat{e}_{i+1}.
\end{eqnarray*}
For (\ref{7.1.1}), by (\ref{1.4}), it is true when $i=0$, and we have
\begin{eqnarray*}
\hat{e}_ie_{i+1}\hat{e}_i&=&e_i\hat{e}_{i-1}(e_ie_{i+1}e_i)\hat{e}_{i-1}e_i\\
&\overset{(\ref{1.3})}{=}&e_i\hat{e}_{i-1}e_i\hat{e}_{i-1}e_i\\
&\overset{induction}{=}&\delta e_i\hat{e}_{i-1}e_i=\delta \hat{e}_i.
\end{eqnarray*}
For (\ref{7.1.1b}), we have
\begin{eqnarray*}
\hat{e}_{i+1}e_i\hat{e}_{i+1}&=& e_{i+1}\hat{e}_i(e_{i+1}e_ie_{i+1})\hat{e}_ie_{i+1}\\
&\overset{(\ref{1.3})}{=}&e_{i+1}(\hat{e}_ie_{i+1}\hat{e}_i)e_{i+1}\overset{(\ref{7.1.1})}{=}\delta e_{i+1}\hat{e}_i e_{i+1}=\delta \hat{e}_{i+1}.
\end{eqnarray*}
For (\ref{7.1.2}), we have
\begin{eqnarray*}
\hat{e}_i^2&=&e_i\hat{e}_{i-1}(e_ie_i)\hat{e}_{i-1}e_i\\
&\overset{(\ref{1.1})}{=}&\delta e_i\hat{e}_{i-1}e_i\hat{e}_{i-1}e_i\\
&\overset{(\ref{7.1.1})}{=}&\delta^2 e_i\hat{e}_{i-1}e_i=\delta^2 \hat{e}_i.\\
\end{eqnarray*}

For (\ref{7.1.3}), we first consider $\hat{e}_i\hat{e}_{i+2}$. Then we have
\begin{eqnarray*}
\hat{e}_i\hat{e}_{i+2}&=&(\hat{e}_ie_{i+2})e_{i+1}\hat{e}_ie_{i+1}e_{i+2}\\
&\overset{(\ref{1.2})}{=}&e_{i+2}(\hat{e}_ie_{i+1}\hat{e}_i)e_{i+1}e_{i+2}\\
&\overset{(\ref{7.1.1})}{=}&\delta e_{i+2}\hat{e}_ie_{i+1}e_{i+2}\\
&\overset{(\ref{1.2})}{=}&\delta e_{i+2}\hat{e}_i(e_{i+2}e_{i+1}e_{i+2})\overset{(\ref{1.3})}{=}\delta\hat{e}_ie_{i+2},
\end{eqnarray*}
and we can obtain the following by Proposition \ref{prop:opp},
$$\hat{e}_{i+2}\hat{e}_{i}=\delta e_{i+2}\hat{e}_i=\delta\hat{e}_ie_{i+2}=\hat{e}_i\hat{e}_{i+2}.$$
Then for the general case of (\ref{7.1.3}), it can be verified by induction and Proposition \ref{prop:opp}.
\\The formula (\ref{7.1.4}) follows from the above proof of (\ref{7.1.3}).
\\ The formula (\ref{7.1.5}) follows from (\ref{7.1.0}),(\ref{7.1.0b}),(\ref{7.1.1}), and(\ref{7.1.1b}).\\
For (\ref{7.1.6}), if $j-i>1$, we have
\begin{eqnarray*}
e_i\hat{e}_j&=&e_j\dots e_{i+2}e_i\hat{e}_{i+1}e_{i+2}\dots e_{j}\\
            &\overset{(\ref{7.1.0b})}{=}&\delta^{-1}e_j\dots e_{i+2}\hat{e}_i\hat{e}_{i+1}e_{i+2}\dots e_{j}\\
            &\overset{(\ref{7.1.4})+\ref{prop:opp}}{=}&\delta^{-1}\hat{e}_i\hat{e}_j;
\end{eqnarray*}
if $i-j>1$, we have
\begin{eqnarray*}
\hat{e}_i\hat{e}_j&\overset{(\ref{7.1.4})}{=}&e_i\dots e_{j+2}(\hat{e}_{j+1}\hat{e}_{j})e_{j+2}\dots e_i\\
                   &\overset{(\ref{7.1.0})+\ref{prop:opp}}{=}&\delta e_i\dots e_{j+2}e_{j+1}\hat{e}_{j}e_{j+2}\dots e_i\\
                   &\overset{(\ref{7.1.4})}{=}&\delta e_i\dots e_{j+2}e_{j+1}e_{j+2}\dots e_i \hat{e}_{j}\\
                   &\overset{(\ref{1.3})}{=}& \delta e_i\hat{e}_{j}.
\end{eqnarray*}

\end{proof}
Let $\psi$ be the algebra morphism from $\RTL(\ddB_n)$ to $\TL(\ddD_{n+1})$.
\begin{lemma} \label{Y1}
The subalgebra $\mathcal{Y}_1$ of $\RTL(\ddB_n)$ generated by $\{e_i\}_{i=1}^{n-1}$ is isomorphic to $\TL(\ddA_{n-1})$.
Then $\mathcal{Y}_1$ has rank $C_n$ over the ground ring.
\end{lemma}
\begin{proof}By easy verification,  there is a surjective algebra morphism $\psi'$ from $\TL(\ddA_{n-1})$  to $\mathcal{Y}_1$.
By $\psi$, we see that $\psi(e_i)=E_{i+2}$, for $1\leq i\leq n-1$, which are Temperley-Lieb generators of  $\TL(\ddD_{n+1})$.
By \cite[Proposition 4]{CW2011}, the subalgebra of $\TL(\ddD_{n+1})$generated by $\{E_{i+2}\}_{i=1}^{n-1}$  is isomorphic to $\TL(\ddA_{n-1})$. Now we have
$$\TL(\ddA_{n-1})\overset{\psi'}{\rightarrow}\mathcal{Y}_1\overset{\psi}{\rightarrow}\TL(\ddD_{n+1}),$$
then $\TL(\ddA_{n-1})\simeq\mathcal{Y}_1.$

\end{proof}
\begin{lemma}\label{Y2}
 The subalgebra $\mathcal{Y}_2$ of $\RTL(\ddB_n)$ generated by $\{\hat{e}_i\}_{i=0}^{n-1}$ can be spanned by $C_{n+1}$ monomials.
\end{lemma}

\begin{proof}We write down the relations about the generators $\{\hat{e}_i\}_{i=0}^{n-1}$ from Lemma \ref{newrel} below,
\begin{eqnarray*}
\hat{e}_ie_{i+1}\hat{e}_i&=&\delta\hat{e}_i,\quad |i-j|=1,\\
\hat{e}_i\hat{e}_j&=&\hat{e}_j\hat{e}_i, \quad |i-j|>1,\\
\hat{e}_i\hat{e}_j\hat{e}_i&=&\delta^2\hat{e}_i\quad |i-j|=1.
\end{eqnarray*}
If we define $\delta=1$,  we can define a surjective algebra morphism  from $\TL(\ddA_n)$ to algebra $\mathcal{Y}_2$.
Then by the monomial reduction of $\TL(\ddA_n)$, the algebra $\mathcal{Y}_2$ can have the same spanning elements as the  reduced monomials of $\TL(\ddA_n)$ with
generators $E_1$, $\dots$, $E_n$
replaced by $e_0$,$\dots$, $e_{n-1}$, respectively. Therefore  $\mathcal{Y}_2$ can be spanned by $C_{n+1}$ monomials, which is the rank of $\TL(\ddA_n)$.
\end{proof}
\begin{prop}\label{rank1}
It follows that
$$\RTL(\ddB_n)=\mathcal{Y}_1+\mathcal{Y}_2.$$
Then $\RTL(\ddB_n)$ can be spanned by $C_n+C_{n+1}-1$ elements.
\end{prop}
\begin{proof}Because the generators of $\RTL(\ddB_n)$ are in $\mathcal{Y}_1+\mathcal{Y}_2$, we just need to prove that
$\mathcal{Y}_1+\mathcal{Y}_2$ are closed under multiplication. Since  $\mathcal{Y}_1$ and $\mathcal{Y}_2$ are algebras
and we have the natural involution in    Proposition \ref{prop:opp}, it remains to prove that  $y_1y_2\in \mathcal{Y}_1+\mathcal{Y}_2$ for $y_1\in \mathcal{Y}_1$
and $y_2\in \mathcal{Y}_2$. We claim that $y_1y_2\in \mathcal{Y}_2$ for  any $y_1\in \mathcal{Y}_1$
and $y_2\in \mathcal{Y}_2$.By induction, it can be reduced to to $e_i\hat{e}_{j}\in \mathcal{Y}_2$. Therefore this holds for
(\ref{7.1.0}),(\ref{7.1.0b}),(\ref{7.1.6}) and Proposition \ref{prop:opp}.
\end{proof}
\section{The Rank of $\RTL(\ddB_n)$}\label{sect:rank}
In this section, we try to prove that the rank of $\RTL(\ddB_n)$ is exactly $C_n+C_{n+1}-1$. To prove this, we mainly use the diagram representation
of Brauer algebra of type $\ddD_{n+1}$ from \cite{CGW2009}. Now we recall the diagram representations here.
Divide $2n+2$ points into two sets $\{1,
2,\ldots, n+1\}$ and $\{\hat{1}, \hat{2}, \ldots, \widehat{n+1}\}$ of points
in the (real) plane with each set on a horizontal line and point $i$ above
$\hat{i}$. An $n+1$-\emph{connector} is a partition on $2n+2$ points into
$n+1$ disjoint pairs. It is indicated in the plane by a (piecewise linear)
curve, called \emph{strand} from one point of the pair to the other. A
\emph{decorated} $n+1$-\emph{connector} is an $n+1$-connector in which an
even number of pairs are labeled $1$, and all other pairs are labeled by
$0$. A pair labeled $1$ will be called \emph{decorated}. The decoration of a
pair is represented by a black dot on the corresponding strand.

\begin{rem}
\label{not:BrMDD}
Denote
$T_{n+1}$ the set of all decorated $n+1$-connectors. Denote $T_{n+1}^0$ the
subset of $T_{n+1}$ of decorated $n+1$-connectors without decorations and
denote $T_{n+1}^=$ the subset of $T_{n+1}$ of decorated $n+1$-connectors
with at least one horizontal strand.

Let $H$ be the commutative monoid with presentation
$$H=\left<\delta^{\pm 1}, \xi, \theta\mid \xi^2=\delta^2, \xi
\theta=\delta\theta, \theta^2=\delta^2\theta\right> =\left<\delta^{\pm
1}\right>\{1, \xi, \theta \}.$$ A \emph{Brauer diagram} of type $\ddD_{n+1}$
is the scalar multiple of a decorated $n$-connector by an element of $H$
belonging to $\left<\delta^{\pm 1}\right>(T_{n+1}\cup \xi T_{n+1}^=\cup
\theta (T_{n+1}^{0}\cap T_{n+1}^=))$.  The \emph{Brauer diagram algebra of
type} $\ddD_{n+1}$, denoted $\BrD(\ddD_{n+1})$, is the
$\Z[\delta^{\pm1}]$-linear span of all Brauer diagrams of type $\ddD_{n+1}$
with multiplication laws defined in \cite[Definition 4.4]{CGW2009}. The
corresponding monoid is denoted
$\BrMD(\ddD_{n+1})$.
\end{rem}

\np
The
scalar $\xi\delta^{-1}$ appears in various products of $n+1$-connectors
described in \cite[Definition 4.4]{CGW2009} and two consecutive black dots
on a strand are removed. Also, the
scalar $\theta\delta^{-1}$ appears in various products of $n+1$-connectors
in which a dotted circle appears, as described in \cite[Figure 16]{CGW2009}.
The multiplication is an intricate variation of
the multiplication in classical Brauer diagrams, where the points of the
bottom of one connector are joined to the points of the top of the other
connector, so as to obtain a new connector. In this process, closed strands
appear which are turned into scalars by translating them into elements of
$H$ as indicated in Figure \ref{ch5closedloop}.\\

\begin{figure}[htbp]
\unitlength 1mm
\begin{picture}(26,15)(0,0)
\linethickness{0.3mm}
\put(10,9.41){\circle{11.18}}
\put(20,10){\makebox(0,0)[cc]{$=$}}
\put(26,10){\makebox(0,0)[cc]{$\delta$,}}
\end{picture}
$\qquad$
\unitlength 1mm
\begin{picture}(36.59,14.18)(0,0)
\put(30.59,9.18){\makebox(0,0)[cc]{$=$}}
\put(36.59,9.18){\makebox(0,0)[cc]{$\theta$,}}
\linethickness{0.3mm}
\put(20.59,8.59){\circle{11.18}}
\put(21.09,3.18){\circle*{2}}
\put(7,8.5){\circle{11.18}}
\put(7.5,3.09){\circle*{2}}
\end{picture}
$\qquad$
\unitlength 1mm
\begin{picture}(27,15)(0,0)
\linethickness{0.3mm}
\put(8.4,12){\circle*{2}}
\put(9,8){\circle*{2}}
\multiput(6.25,5.62)(0.08,0.08){100}{\line(0,1){0.09}}
\multiput(6.25,13.75)(0.08,-0.08){96}{\line(1,0){0.09}}
\put(20,10){\makebox(0,0)[cc]{$=$}}
\put(27,10){\makebox(0,0)[cc]{$\xi$}}
\multiput(6.25,13.75)(0.08,0.00){94}{\line(0,1){0.09}}
\multiput(6.25,5.62)(0.08,0.00){94}{\line(0,1){0.09}}
\end{picture}
\caption{The closed loops corresponding to the generators of $H$}
\label{ch5closedloop}
\end{figure}
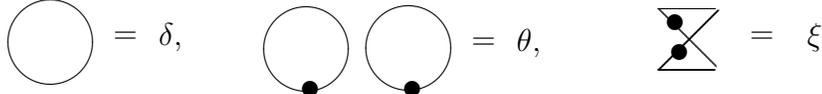

\begin{figure}[!htb]\label{psi}
\begin{center}
\includegraphics[width=.9\textwidth, height=.26\textheight]{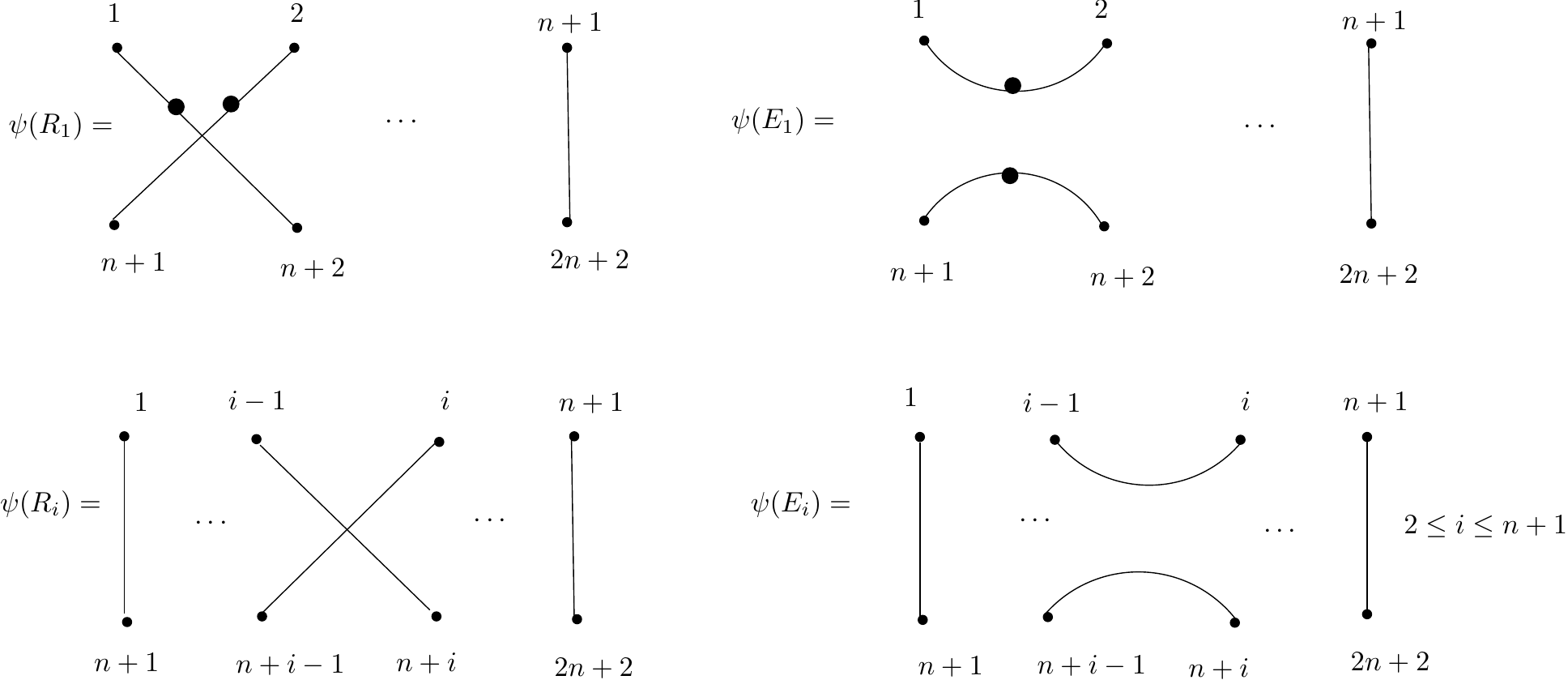}
\end{center}
\caption{The images of the generators of $\Br(\ddD_{n+1})$ under $\psi$}
\label{fig:1}
\end{figure}
Let $c_1$ and $c_2$ decorated $n+1$-connectors, and $\kappa_1$ $\kappa_2\in \{1, \xi, \theta\}$.
Now we describe the product $\kappa_1c_1\kappa_2 c_2$ in \cite[Definition 4.4]{CGW2009} being the form of
$\kappa c$ where $c$ is a decorated $n+1$-connector and $\kappa\in H$.\\
\begin{enumerate}[(i)]
\item As the classical case, draw the diagram $c_1$ and $c_2$, and stack them.
\item Determine the pairing of $c$: for a point at the top  of $c_1$ or the bottom of $c_2$, follow the strand until it ends
in a point at the top of $c_1$ or the bottom of $c_2$. This results in a new pairing for $c$.
\item set $\kappa=\kappa_1\kappa_2$. For each  straightening step in a concatenation of pairs
as carried out in the previous step, check if the pattern shrunk to a straight horizontal line segment occurs as the left hand-sides of the first 20 relations in Figure
\ref{22rels}. If so, multiply $\kappa$ by $\xi\delta^{-1}$; otherwise, $\kappa$ is not changed.
(Compare with the left-hand picture of the last two relations in  Figure \ref{22rels}; this pattern as well as each triple of straight line segments forming a shape appearing in the first 20 relations in Figure \ref{22rels} but whose decoration pattern does not appear in the first 20 relations in
Figure \ref{22rels}, does not change  change $\kappa$.)
\item At this stage, only closed loops remains. closed loops come from strands which have no endpoints in $c$. First simplify loops
by removing crossing as in (iii), i.e. by use of the first 20 relations in Figure \ref{22rels}(again, the configurations not appearing in the figure do not give $\xi\delta^{-1}$) and shrink them using the rules on the bottom lines of the first 20 relations in  Figure \ref{22rels}
(at this stage, factor $\xi\delta^{-1}$ may emerge). Next, replace each closed loops without decoration by $\delta$ (that is, remove the loop and multiply $\kappa$ by $\delta$) and each pair of disjoint closed decoration loops by $\theta$. As the number of decorated pairs is even, what might remain is a simple decorated loop in the presence of a decorated pair; if so , undecorate the pair, remove the decorated loop by multiply
$\kappa$ by $\theta\delta^{-1}$. (Compare with the right-hand side of the last two relations in Figure \ref{22rels}.)
\item If $\theta$ is a factor of $\kappa$, remove all decorations from $c$.
\begin{figure}[!htb]
\begin{center}
\includegraphics[width=.9\textwidth, height=.8\textheight]{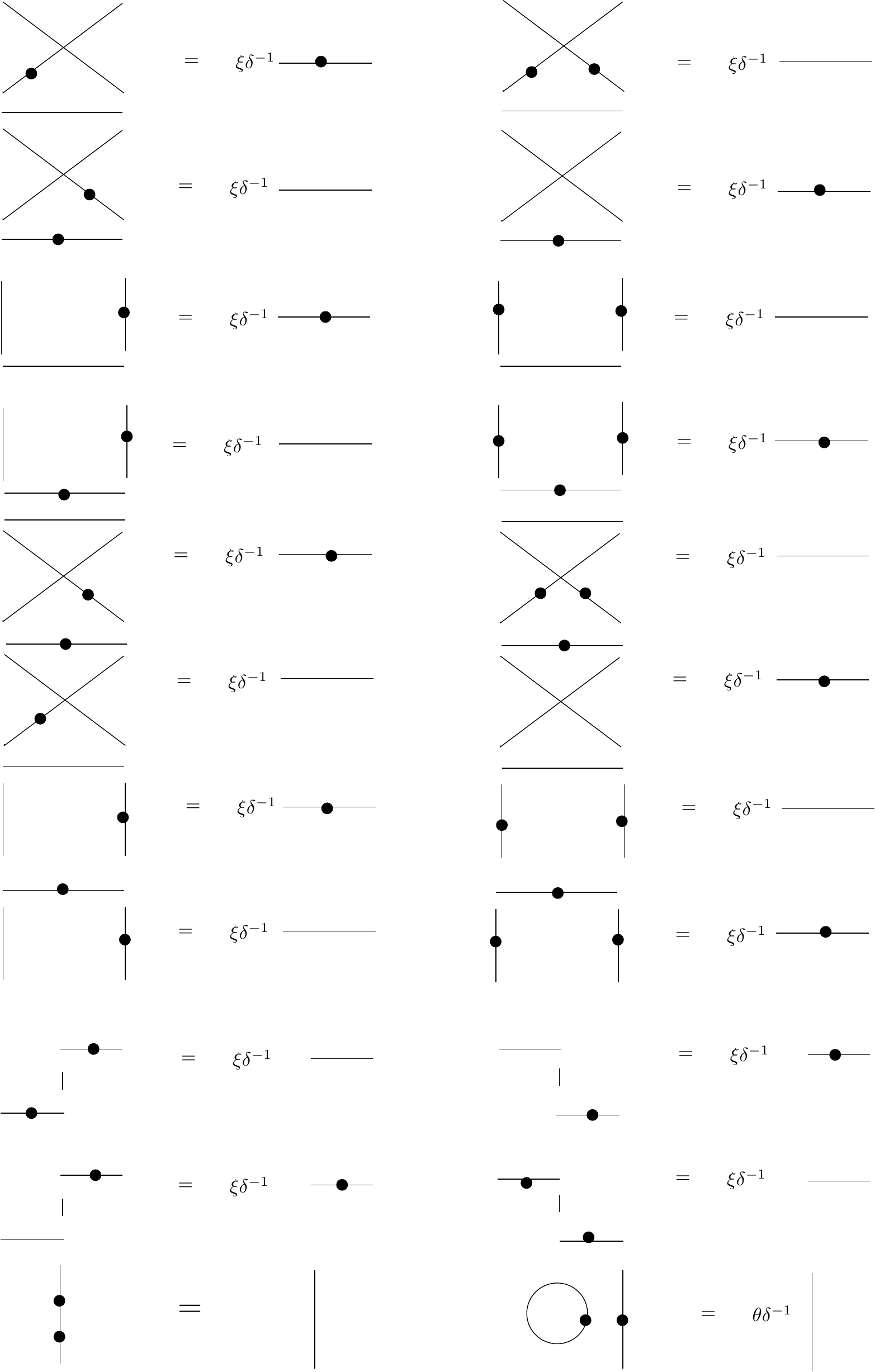}
\end{center}
\caption{22 reduction relations for Brauer diagram algebra of type $\ddD_{n+1}$}
\label{22rels}
\end{figure}
\end{enumerate}
In \cite{CGW2009}, the algebra $\BrD(\ddD_{n+1})$ is proved to be isomorphic
to $\Br(\ddD_{n+1})$ by means of the isomorphism
$\psi:\Br(\ddD_{n+1})\mapsto \BrD(\ddD_{n+1}) $ defined on generators as
in Figure \ref{fig:1}.  It is free over $\Z[\delta^{\pm 1}]$ with basis
$T_{n+1}\cup \xi T_{n+1}^=\cup \theta (T_{n+1}^{0}\cap T_{n+1}^=)$.
\begin{thm} The algebra $\RTL(\ddB_n)$ has rank $C_n+C_{n+1}-1$.
\end{thm}
\begin{proof}Suppose that the canonical Temperley-Lieb basis for $\mathcal{Y}_1$ in lemma \ref{Y1} is $\mathcal{K}_1$
and the spanning set for   $\mathcal{Y}_2$ in lemma \ref{Y2} is $\mathcal{K}_2$. We see that $\psi\phi(e_i)=\psi(E_{i+2})$,
which is drawn in Figure \ref{psi} for $i=1,\dots, n-1$.  Then we see those $\{\psi(E_{i+2})\}_{i=1}^{n-1}$ generate the canonical Temperley-Lieb algebra of
type $\ddA_{n-1}$, which has rank $C_n$ , and $\mathcal{K}_1\subset T_{n+1}$, which represents $C_n$ different diagrams.
For $\mathcal{K}_2\setminus\{1\}$, we have $\mathcal{K}_2\setminus\{1\}\subset \theta (T_{n+1}^{0}\cap T_{n+1}^=)$, up to some powers of $\delta$. Because
\begin{eqnarray*}
\psi\phi(\hat{e}_i)&=&\theta \delta\psi(E_{i+2}),\quad  for \quad i=1,\dots, n-1, \\
\psi\phi(\hat{e}_0)&=&\theta \delta\psi(E_{1}),
\end{eqnarray*}
which are generators of $\TL(\ddA_n)$, without considering $\theta$ and $\delta$. So  $\mathcal{K}_2\setminus\{1\}\subset \theta (T_{n+1}^{0}\cap T_{n+1}^=)$, and
represents $C_{n+1}-1$ different diagrams in $\theta (T_{n+1}^{0}\cap T_{n+1}^=)$.
Therefore $\RTL(\ddB_n)$ has rank $C_n+C_{n+1}-1$.
\end{proof}
Using the cellular structure of $\Br(\ddB_n)$ (\cite{CL2011}) and $\Br(\ddC_n)$ (\cite{CLY2010}), we can obtain the following Theorem.
\begin{thm}
The algebras $\RTL(\ddB_n)$ and $\RTL(\ddC_n)$ are cellular algebras in the sense of \cite{GL1996}.
\end{thm}
\begin{rem}
By \cite{F1997}, it is known that the Temperley-Lieb algebras of type $\ddF_n$ ,  $\TL(\ddF_n)$ for $n\geq 4$ is of finite rank.
We can verify that the $\RTL(\ddF_n)$ is a quotient algebra of $\TL(\ddF_n)$, so  $\RTL(\ddF_n)$  is  also of finite rank. But it is not easy to give the precise rank  here,
we will leave it for some further research.
\end{rem}

Shoumin Liu\\
Email: s.liu@sdu.edu.cn\\
School of mathmatics, Shandong University\\
Shanda Nanlu 27, Jinan, \\
Shandong Province, China\\
Postcode: 250100


\begin{thebibliography}{9}
\bibitem{AM2006}M. Alvarez, P.P. Martin, A Temperley-Lieb category for $2$-manifold,
arxiv 0711.4777.
\bibitem{Bigelow2001}
S. Bigelow, Braid groups are linear,
Journal of the American Mathematical Society 14(2001), 471-486.
\bibitem{Brauer1937}R.Brauer,
On algebras which are connected with the semisimple continous groups,
Annals of mathematics, {\bf 38}, 1937,857--872.
\bibitem{Car}
R. Carter, Simple group of Lie type, 1989 by John Wiley \& sons Ltd.
\bibitem{CohenWales2002}
A. M. Cohen, David. B. Wales, Linearity of Artin groups of finite types,
Israel Journal of Mathematics, 131 (2002),101--123.
\bibitem{CLY2010}A.M.~Cohen, S.~Liu and S.~Yu, Brauer algebras of type C,
Journal of Pure and Applied Algebra, {\bf 216} (2012), 407--426.

\bibitem{CL2011}A.M.~Cohen, S.~Liu, Brauer algebras of type B,
Forum Mathematicum, {\bf 27} (2015),1163--1202.

\bibitem{CW2011}
A.M.~Cohen and D.B.~Wales, The Birman-Murakami-Wenzl algebras of type $\ddE_n$, {\rm Transformation Groups}, {\bf 16} (2011),
681--715.
\bibitem{Crisp1996}
J. Crisp, Injective maps between Artin groups, in Geometric Group Theory Down Under,
Lamberra 1996 (J. Cossey, C.F. Miller III, W.D. Neumann and M.Shapiro, eds.)
De Gruyter, Berlin, 1999, 119-137.
\bibitem{CFW2008}A.M.~Cohen, B.~Frenk and D.B.~Wales, Brauer algebras of simply laced type,
Israel J. Math. {\bf 173} (2009), 335--365.
\bibitem{CGW2006}A.M.~Cohen, Di\'{e} A.H.~Gijsbers and D.B.~Wales,
A poset connected to Artin monoids of simply laced type,
Journal of Combinatorial Theory, Series A {\bf 113} (2006), 1646--1666.
\bibitem{CGW2009}A.M.~Cohen, Di\'{e} A.H.~Gijsbers and D.B.~Wales,
Tangle and Brauer diagram algebras of type $\ddD_{n}$, Journal of  Knot Theory and its Ramifications,
 {\bf 18} (2009), 447--483.
\bibitem{Dieck}
T.tom Dieck, Symmetrische Br$\ddot{u}$cken und Knotentheorie zu den Dykin-Diagramen vom Typ B,
 http://www.uni-math.gwdg.de/tammo/preprints/tb.pdf.
\bibitem{Dieck2003}
T. tom Dieck, Quantum groups and knot algebra, Lecture notes, May, 2004,http://www.uni-math.gwdg.de/tammo/dm.pdf.
\bibitem{Digne2003}
F.~Digne, On the Linearity of Artin Braid groups,
Journal of Algebra {\bf 268} (2003), 39-57.
\bibitem{F1997} K.~Fan, Structure of a Heck algebra quotient, Journal of American mathematical society,
{\bf 10}(1997), 139--167.
\bibitem{G2007}
 M. Geck, Heck algebra of finite type are cellular, Inventiones Mathematiccae 169(2007),
501-517.
\bibitem{GL1996}
J.J. Graham and G.I. Lehrer, Cellular algebras, Inventiones Mathematiccae 123(1996),
1-44.
\bibitem{KMY2014}
Z.Kadar, P.P. Martin, S.Yu, On geometrically defined extensions of the
Temperley¨CLieb category in the Brauer category, arxiv 1401.1774v1.
\bibitem{Lthesis}S.~Liu, Brauer algebras of non-simply laced type, PhD thesis, Technische Universiteit Eindhoven, 2012.
\bibitem{L2013}S.~Liu, Brauer algebra of type $\ddF_4$,
Indagationes Mathematicae, {\bf 24} (2013), 428--442.
\bibitem{L2015}S.~Liu, Brauer algebra of multiply laced Weyl type,
Indagationes Mathematicae, {\bf 26} (2015), 526--546.

\bibitem{TL1971} H.N.V. Temperley and E. Lieb, Relation between
percolation and colouring problems and other graph theoretical problems
associated with regular planar lattices: some exact results for the
percolation problems, Proc.~R.~Soc.~A322 (1971) 251-288.
\bibitem{T1959}
J.Tits,
Groupes alg\'ebriques semi-simples et g\'eom\'etries associ\'ees. In Algebraic
and topological Foundations of Geometry (Proc. Colloq., Utrecht, 1959).
Pergamon, Oxford, 175--192.
\end{thebibliography}
\end{document}